\documentclass[12pt, twoside]{article}
\usepackage{amsmath,amsthm,amssymb}
\usepackage{times}
\usepackage{enumerate}

\def\titlerunning#1{\gdef\titrun{#1}}
\makeatletter
\def\author#1{\gdef\autrun{\def\and{\unskip, }#1}\gdef\@author{#1}}
\def\address#1{{\def\and{\\\hspace*{18pt}}\renewcommand{\thefootnote}{}%
\footnote {#1}}%
\markboth{\autrun}{\titrun}}
\makeatother
\def\email#1{e-mail: #1}
\def\subjclass#1{{\renewcommand{\thefootnote}{}%
\footnote{\emph{Mathematics Subject Classification $(2010)$:} #1}}}
\def\keywords#1{\par\medskip
\noindent\textbf{Keywords.} #1}

\newtheorem{thm}{Theorem}[section]

\newtheorem{lem}[thm]{Lemma}

\newtheorem{proposition}[thm]{Proposition}

\theoremstyle{definition}
\newtheorem{defin}[thm]{Definition}
\newtheorem{rem}[thm]{Remark}

\newtheorem{open}{Open Problem}

\numberwithin{equation}{section}

\def\Plambda{(1.15)_{\lambda}}
\def\Pext{(1.15)_{\lambda^\star}}

\begin{document}

%%%%% To ease editing, add:

\baselineskip=17pt

%%%%%%%%%%%%%%%%

%% In the running head, give an abbreviation of the title.
\titlerunning{$p$-Laplace equations and geometric Sobolev inequalities}

\title{Regularity of stable solutions of $p$-Laplace equations through geometric
Sobolev type inequalities}

\author{Daniele Castorina
\and
Manel Sanch\'on}

\date{}

\maketitle

\address{D. Castorina: Departament de Matem\`atiques,
Universitat Aut\`onoma de Barcelona,
08193 Bellaterra, Spain; \email{castorina@mat.uab.cat}
\and
M. Sanch\'on: Departament de Matem\`atica Aplicada i An\`alisi,
Universitat de Barcelona,
Gran Via 585, 08007 Barcelona, Spain; \email{msanchon@maia.ub.es}}


%%%%%%%%

\begin{abstract}
In this paper we prove a
Sobolev and a Morrey type inequality involving the mean
curvature and the tangential gradient with respect to the level sets
of the function that appears in the inequalities. Then, as an
application, we establish \textit{a priori} estimates for semi-stable
solutions of $-\Delta_p u= g(u)$ in a smooth bounded domain
$\Omega\subset \mathbb{R}^n$. In particular, we obtain new
$L^r$ and $W^{1,r}$ bounds for the extremal solution
$u^\star$ when the domain is strictly convex. More precisely, we prove that
$u^\star\in L^\infty(\Omega)$ if $n\leq p+2$ and $u^\star\in
L^{\frac{np}{n-p-2}}(\Omega)\cap W^{1,p}_0(\Omega)$ if $n>p+2$.
%% Keywords are optional
\end{abstract}

\section{Introduction}
The aim of this paper is to obtain \textit{a priori} estimates for
semi-stable solutions of $p$-Laplace equations. We will accomplish this by proving
some geometric type inequalities involving the functionals
\begin{equation}\label{Ipq}
I_{p,q}(v;\Omega):=\left( \int_{\Omega}
\Big(\frac{1}{p'}|\nabla_{T,v} |\nabla v|^{p/q}|\Big)^{q}
+ |H_v|^q |\nabla v|^p \, dx \right)^{1/p},\quad p,q\geq 1
\end{equation}
where $\Omega$ is a smooth bounded domain of $\mathbb{R}^n$ with
$n\geq 2$ and $v\in C_0^\infty(\overline{\Omega})$. Here, and in the
rest of the paper, $H_v (x)$ denotes the mean curvature at $x$ of
the hypersurface $\{y\in\Omega:|v(y)|=|v(x)|\}$ (which is smooth at
points $x\in\Omega$ satisfying $\nabla v(x)\neq 0$), and $\nabla_{T,v}$
is the tangential gradient along a level set of $|v|$. We will prove a Morrey's type
inequality when $n<p+q$ and a Sobolev inequality when $n>p+q$ (see
Theorem~\ref{Theorem:Sobolev} below).

Then, as an application of these inequalities, we establish
$L^r$ and $W^{1,r}$ \textit{a priori} estimates for semi-stable solutions of
the reaction-diffusion problem
\begin{equation}\label{problem}
\left\{
\begin{array}{rcll}
-\Delta_p u &=& g(u)  &\textrm{in } \Omega, \\
 u&>& 0  &\textrm{in } \Omega, \\
 u &=& 0  &\textrm{on } \partial \Omega.
\end{array}
\right.
\end{equation}
Here, the diffusion is modeled
by the $p$-Laplace operator $\Delta_p$
(remember that $\Delta_p u:= {\rm div}(|\nabla u|^{p-2}\nabla u)$)
with $p>1$, while the reaction term is driven by any positive $C^1$
nonlinearity $g$.

As we will see, these estimates will lead to new $L^r$ and $W^{1,r}$ bounds
for the extremal solution $u^\star$ of \eqref{problem} when $g(u)=\lambda f(u)$
and the domain $\Omega$ is strictly convex. More precisely, we prove that
$u^\star\in L^\infty(\Omega)$ if $n\leq p+2$ and $u^\star\in
L^{\frac{np}{n-p-2}}(\Omega)\cap W^{1,p}_0(\Omega)$ if $n>p+2$.

\subsection{Geometric Sobolev inequalities}

Before we establish our Sobolev and Morrey type inequalities we will
state that the functional $I_{p,q}$ defined in \eqref{Ipq} decreases
(up to a universal multiplicative constant) by Schwarz
symmetrization. Given a Lipschitz continuous function $v$ and its
Schwarz symmetrization $v^*$ it is well known that
$$
\int_{B_R} |v^*|^r\ dx=\int_\Omega |v|^r\ dx\quad \textrm{for all }r\in[1,+\infty]
$$
and
$$
\int_{B_R} |\nabla v^*|^r\ dx \leq \int_\Omega |\nabla v|^r\ dx
\quad \textrm{for all }r\in[1,\infty).
$$

Our first result establishes that $I_{p,q}(v^*;B_R)\leq C I_{p,q}(v;\Omega)$ for
some universal constant $C$ depending only on $n$, $p$, and $q$.

%------------------------------------------------------------------------------------
%------------------------------------------------------------------------------------
\begin{thm}\label{thm:Ipq}
Let $\Omega$ be a smooth bounded domain of $\mathbb{R}^n$ with $n\geq 2$ and
$B_R$ the ball centered at the origin and with radius $R=(|\Omega|/|B_1|)^{1/n}$.
Let $v\in C^\infty_0(\overline{\Omega})$ and $v^*$ its Schwarz symmetrization.
Let $I_{p,q}$ be the functional defined in
\eqref{Ipq} with $p,q \geq 1$. If $n>q+1$ then there exists a universal constant $C$ depending
only on $n$, $p$, and $q$, such that
\begin{equation}\label{comp_integrals}
\left( \int_{B_R}\frac{1}{|x|^q} |\nabla v^*|^p \, dx \right)^{1/p}
=
I_{p,q}(v^*;B_R)
\leq
C I_{p,q}(v;\Omega).
\end{equation}
\end{thm}
%------------------------------------------------------------------------------------
%------------------------------------------------------------------------------------

Note that the Schwarz symmetrization of $v$ is a radial function,
and hence, its level sets are spheres. In particular, the mean
curvature $H_{v^*}(x)=1/|x|$ and the tangential gradient
$\nabla_{T,{v^*}} |\nabla v^*|^{p/q}=0$. This explains the equality
in \eqref{comp_integrals}.

A related result was proved by Trudinger \cite{Trudinger97} when
$q=1$ for the class of mean convex functions (\textit{i.e.}, functions
for which the mean curvature of the level sets is nonnegative). More precisely,
he proved Theorem~\ref{thm:Ipq} replacing the functional $I_{p,q}$ by
\begin{equation}\label{tilde:Ipq}
\tilde{I}_{p,q}(v;\Omega):=\left( \int_{\Omega}
|H_v|^q |\nabla v|^p \, dx \right)^{1/p}
\end{equation}
and considering the Schwarz symmetrization of $v$ with respect to the
perimeter instead of the classical one like us (see Definition~\ref{Schwarz-symm}
below). In order to define
this symmetrization (with respect to the perimeter) it is essential
to know that the mean curvature $H_v$ of the level sets of $|v|$ is
nonnegative. Then using an Aleksandrov-Fenchel inequality for mean
convex hypersurfaces (see \cite{Trudinger94}) he proved Theorem~\ref{thm:Ipq}
for this class of functions when $q=1$.

We prove Theorem~\ref{thm:Ipq} using two ingredients. The first one
is the classical isoperimetric inequality:
\begin{equation}\label{isop:ineq}
n|B_1|^{1/n}|D|^{(n-1)/n}\leq |\partial D|
\end{equation}
for any smooth bounded domain $D$ of $\mathbb{R}^n$. The second one is
a geometric Sobolev inequality, due to Michael and Simon \cite{MS}
and to Allard \cite{A}, on compact $(n-1)$-hypersurfaces $M$ without
boundary which involves the mean curvature $H$ of $M$:
for every $q\in  [1, n-1)$,
there exists a constant $A$ depending only on $n$ and $q$ such that
\begin{equation}\label{Sob:mean}
\left( \int_M |\phi|^{q^\star} d\sigma \right)^{1/q^\star} \leq
A
\left( \int_M |\nabla \phi|^q +  |H\phi|^q \ d\sigma \right)^{1/q}
\end{equation}
for every $\phi\in C^\infty(M)$, where $q^\star = (n-1)q/(n-1-q)$ and
$d\sigma$ denotes the area element in $M$.
Using the classical isoperimetric inequality \eqref{isop:ineq} and
the geometric Sobolev inequality \eqref{Sob:mean} with $M=\{x\in \Omega:|v(x)|=t\}$
and $\phi=|\nabla v|^{(p-1)/q}$ we will prove Theorem~\ref{thm:Ipq}
with the explicit constant $C=A^\frac{q}{p} |\partial B_1|^\frac{q}{(n-1)p}$, being $A$ the
universal constant in \eqref{Sob:mean}.

{F}rom Theorem~\ref{thm:Ipq} and well known 1-dimensional weighted
Sobolev inequalities it is easy to prove Morrey and Sobolev geometric
inequalities involving the functional $I_{p,q}$. Indeed, by
Theorem~\ref{thm:Ipq} and since Schwarz symmetrization preserves the
$L^r$ norm, it is sufficient to prove the existence of a positive
constant $\overline{C}$ independent of $v^*$ such that
$$
\|v^*\|_{L^r(B_R)}\leq \overline{C} I_{p,q}(v^*;B_R).
$$

Using this argument we prove the following geometric inequalities.

%------------------------------------------------------------------------------------
%------------------------------------------------------------------------------------
%------------------------------------------------------------------------------------
\begin{thm}\label{Theorem:Sobolev}
Let $\Omega$ be a smooth bounded domain of $\mathbb{R}^n$ with $n\geq 2$ and
$v\in C^\infty_0(\overline{\Omega})$.
Let $I_{p,q}$ be the functional defined in \eqref{Ipq} with $p,q \geq 1$ and
$$
p_{q}^\star:= \frac{n p}{n - (p+q)}.
$$
Assume $n>q+1$. The following assertions hold:
\begin{enumerate}
\item[$(a)$] If $n<p+q$ then
\begin{equation}\label{Morrey}
\|v\|_{L^\infty(\Omega)}
\leq C_1|\Omega|^{\frac{p+q-n}{np}}I_{p,q}(v;\Omega)
\end{equation}
for some constant $C_1$ depending only on $n$, $p$, and $q$.

\item[$(b)$] If $n>p+q$, then
\begin{equation}\label{Sobolev}
\|v\|_{L^r(\Omega)}
\leq C_2|\Omega|^{\frac{1}{r}-\frac{1}{p_q^\star}}
I_{p,q}(v;\Omega)\quad \textrm{for every }1\leq r \leq p_{q}^\star,
\end{equation}
where $C_2$ is a constant depending only on
$n$, $p$, $q$, and $r$.

\item[$(c)$] If $n = p+q$, then
\begin{equation}\label{Moser-Trudinger}
\int_\Omega\exp\left\{\left(\frac{|v|}{C_3 I_{p,q}(v;\Omega)}\right)^{p'}\right\}\ dx
\leq \frac{n}{n-1}|\Omega|,\quad \textrm{where }p'=p/(p-1),
\end{equation}
for some positive constant $C_3$ depending only on $n$ and $p$.
\end{enumerate}
\end{thm}

Cabr\'e and the second author \cite{CS} proved recently
Theorem~\ref{Theorem:Sobolev} under the assumption $q\geq p$ using a
different method (without the use of Schwarz symmetrization). More
precisely, they proved the theorem replacing the functional
$I_{p,q}(v;\Omega)$ by the one defined in \eqref{tilde:Ipq},
$\tilde{I}_{p,q}(v;\Omega)$. Therefore, our
geometric inequalities are only new in the range $1\leq q<p$.

\begin{open}
Is Theorem~\ref{Theorem:Sobolev} true for the range $1\leq q<p$ and
replacing the functional $I_{p,q}(v;\Omega)$ by the one defined in
\eqref{tilde:Ipq}, $\tilde{I}_{p,q}(v;\Omega)$?
\end{open}

This question has a posive answer for the class of mean convex functions.
Trudinger \cite{Trudinger97} proved this result for this class of functions
when $q=1$ and can be easily extended for every $q\geq 1$. However, to our
knowledge, for general functions (without mean convex level sets) it is an
open problem.

%------------------------------------------------------------------------------------
%------------------------------------------------------------------------------------
%------------------------------------------------------------------------------------
%------------------------------------------------------------------------------------
%------------------------------------------------------------------------------------
%------------------------------------------------------------------------------------
%------------------------------------------------------------------------------------
%------------------------------------------------------------------------------------
%------------------------------------------------------------------------------------
%------------------------------------------------------------------------------------
\subsection{Regularity of semi-stable solutions}

The second part of the paper deals with \textit{a priori} estimates for semi-stable
solutions of problem \eqref{problem}. Remember that a regular solution $u\in
C_0^1(\overline{\Omega})$ of \eqref{problem} is said to be \textit{semi-stable}
if the second variation of the associated energy functional at $u$ is nonnegative
definite, \textit{i.e.},
\begin{equation}\label{semi-stab1}
\int_\Omega |\nabla u|^{p-2} \left\{|\nabla \phi|^2+(p-2)
\left(\nabla \phi\cdot\frac{\nabla u}{|\nabla u|}\right)^2\right\} - g'(u) \phi^2\ dx \geq 0
\end{equation}
for every $\phi \in H_0$, where $H_0$ denotes the space of admissible functions
(see Definition~\ref{H0} below). The class of semi-stable solutions
includes local minimizers of the energy functional as well as minimal
and extremal solutions of \eqref{problem} when $g(u)=\lambda f(u)$.

Using an appropriate test function in \eqref{semi-stab1} we prove
the following \textit{a priori} estimates for semi-stable solutions.
This result extends the ones in \cite{Cabre09} and \cite{CS} for the
Laplacian case ($p=2$) due to Cabr\'e and the second author.
%------------------------------------------------------------------------------------
%------------------------------------------------------------------------------------
%------------------------------------------------------------------------------------
\begin{thm}\label{Theorem}
Let $g$ be any $C^\infty$ function and $\Omega\subset\mathbb{R}^n$ any smooth
bounded domain. Let $u\in C^1_0(\overline{\Omega})$ be a semi-stable
solution of \eqref{problem}, \textit{i.e.}, a solution satisfying \eqref{semi-stab1}.
The following assertions hold:

$(a)$ If $n\leq p+2$ then there exists a constant $C$ depending
only on $n$ and $p$ such that
\begin{equation}\label{L-infinty}
\|u\|_{L^\infty(\Omega)}\leq s+\frac{C}{s^{2/p}}|\Omega|^\frac{p+2-n}{np}
\left(\int_{\{u\leq s\}}  |\nabla u|^{p+2}\, dx\right)^{1/p}\quad \textrm{for
all }s>0.
\end{equation}

$(b)$ If $n>p+2$ then there exists a constant $C$ depending
only on $n$ and $p$ such that
\begin{equation}\label{Lq:estimate}
\left(\int_{\{u>s\}} \Big(|u|-s\Big)^{\frac{np}{n-(p+2)}}\ dx\right)^{\frac{n-(p+2)}{np}}
\leq \frac{C}{s^{2/p}}
\left(\int_{\{u\leq s\}} |\nabla u|^{p+2} \ dx\right)^{1/p}
\end{equation}
for all $s>0$. Moreover, there exists a constant $C$ depending
only on $n$, $p$, and $r$ such that
\begin{equation}\label{grad:estimate}
\int_\Omega |\nabla u|^r\ dx\leq C\left(|\Omega|+\int_\Omega|u|^\frac{np}{n-(p+2)}\ dx
+\|g(u)\|_{L^1(\Omega)}\right)
\end{equation}
for all $1\leq r<r_1:=\frac{np^2}{(1+p)n-p-2}$.
\end{thm}
%************************************************************************

To prove \eqref{L-infinty} and \eqref{Lq:estimate} we use the semi-stability condition
\eqref{semi-stab1} with the test function $\phi=|\nabla u|\eta$ to obtain
\begin{equation}\label{ineq:key}
\int_{\Omega}\left( \frac{4}{p^2}|\nabla_{T,u} |\nabla u|^{p/2}|^{2}
+ \frac{n-1}{p-1}H_u^2 |\nabla u|^{p} \right) \eta^2 \, dx
\leq \int_{\Omega}  |\nabla u|^{p}|\nabla \eta|^2\, dx
\end{equation}
for every Lipschitz function $\eta$ in $\overline{\Omega}$ with
$\eta|_{\partial\Omega}=0$. Then, taking $\eta=T_s u = \min\{s,u\}$,
we obtain \eqref{L-infinty} and \eqref{Lq:estimate} when $n\neq p+2$
by using the Morrey and Sobolev inequalities established in
Theorem~\ref{Theorem:Sobolev} with $q=2$.
The critical case $n=p+2$ is more involved. In order to get
\eqref{L-infinty} in this case, we take another
explicit test function $\eta=\eta(u)$ in \eqref{ineq:key} and use the geometric
Sobolev inequality \eqref{Sob:mean}. The
gradient estimate established in \eqref{grad:estimate} will follow by
using a technique introduced by B\'enilan \textit{et al.}
\cite{BBGGPV95} to get the regularity of entropy solutions for
$p$-Laplace equations with $L^1$ data (see Proposition
\ref{Prop:bootstrap}).

%---------------------------------------------------------------------------
%---------------------------------------------------------------------------
%---------------------------------------------------------------------------
%---------------------------------------------------------------------------
%---------------------------------------------------------------------------
The rest of the introduction deals with the regularity of extremal solutions.
Let us recall the problem and some known results in this topic. Consider
\stepcounter{equation}
$$
\left\{
\begin{array}{rcll}
-\Delta_p u&=&\lambda f(u)&\textrm{in }\Omega,\\
u&=&0&\textrm{on }\partial \Omega,
\end{array}
\right. \eqno{(1.15)_{\lambda}}
$$
where $\lambda$ is a positive parameter and $f$ is a $C^1$ positive
increasing function satisfying
\begin{equation}\label{p-superlinear}
\lim_{t\rightarrow+\infty}\frac{f(t)}{t^{p-1}}=+\infty.
\end{equation}

Cabr\'e and the second author \cite{CS07} proved the existence of
an extremal parameter $\lambda^\star\in(0,\infty)$ such that problem
$\Plambda$ admits a minimal regular solution $u_\lambda\in C^1_0(\overline{\Omega})$
for $\lambda\in(0,\lambda^\star)$ and admits no regular solution
for $\lambda>\lambda^\star$.
Moreover, every minimal solution $u_\lambda$ is a semi-stable for $\lambda\in
(0,\lambda^\star)$.

For the Laplacian case ($p=2$), the limit of minimal solutions
$$
u^\star:=\lim_{\lambda\uparrow\lambda^\star}u_\lambda
$$
is a weak solution of the extremal problem $\Pext$ and it is known as
extremal solution. Nedev \cite{Nedev} proved, in the case of convex
nonlinearities, that $u^\star\in L^\infty(\Omega)$ if $n\leq 3$ and
$u^\star\in L^r(\Omega)$ for all $1\leq r<n/(n-4)$ if $n\geq 4$.
Recently, Cabr\'e~\cite{Cabre09}, Cabr\'e and the second author \cite{CS}, 
and Nedev \cite{Nedev01} proved,
in the case of convex domains and general nonlinearities, that $u^\star\in L^\infty(\Omega)$
if $n\leq 4$ and $u^\star\in L^{\frac{2n}{n-4}}(\Omega)\cap H^1_0(\Omega)$
if $n\geq 5$.

For arbitrary $p>1$ it is unknown if the limit of minimal solutions
$u^\star$ is a (weak or entropy) solution of $\Pext$. In the affirmative
case, it is called the \textit{extremal solution of $\Pext$}.
However, in \cite{S} it is proved  that the limit of minimal solutions $u^\star$ is
a weak solution (in the distributional sense) of $\Pext$ whenever $p\geq 2$ and
$f$ satisfies the additional condition:
\begin{equation}\label{convex:assump}
\textrm{there exists }T\geq0 \textrm{ such that }(f(t)-f(0))^{1/(p-1)}
\textrm{ is convex for all }t\geq T.
\end{equation}
Moreover,
$$
u^\star\in L^\infty(\Omega)\qquad\textrm{ if }n<p+p'
$$
and
$$
u^\star\in L^r(\Omega),\textrm{ for all }r<\tilde{r}_0:=(p-1)\frac{n}{n-(p+p')},\quad
\textrm{if }n\geq p+p'.
$$
This extends previous results of Nedev \cite{Nedev} for the Laplacian case ($p=2$)
and convex nonlinearities.

Our next result improves the $L^q$ estimate in \cite{Nedev,S} for strictly convex domains. We also
prove that $u^\star$ belongs to the energy class $W^{1,p}_0(\Omega)$ independently
of the dimension extending an unpublished result of Nedev~\cite{Nedev01} for $p=2$ to every
$p\geq 2$ (see also \cite{CS}).
%------------------------------------------------------------------------------------
%------------------------------------------------------------------------------------
%------------------------------------------------------------------------------------
\begin{thm}\label{Theorem2}
Let $f$ be an increasing positive $C^1$ function satisfying
\eqref{p-superlinear}. Assume that $\Omega$ is a  smooth strictly convex
domain of $\mathbb{R}^n$. Let $u_\lambda\in
C^1_0(\overline{\Omega})$ be the minimal solution of $\Plambda$.
There exists a constant $C$ independent of $\lambda$ such that:
\begin{enumerate}
\item[$(a)$] If $n\leq p+2$ then 
$
\|u_\lambda\|_{L^\infty(\Omega)}\leq C \|f(u_\lambda)\|_{L^1(\Omega)}^{1/(p-1)}.
$
\item[$(b)$] If $n> p+2$ then
$
\|u_\lambda\|_{L^\frac{np}{n-p-2}(\Omega)}\leq C
\|f(u_\lambda)\|_{L^1(\Omega)}^{1/(p-1)}.
$ %\quad\textrm{and}\quad
Moreover
$
\|u_\lambda\|_{W^{1,p}_0(\Omega)}\leq C'
$ 
where $C'$ is a constant depending only on $n$, $p$, $\Omega$, $f$ and 
$\|f(u_\lambda)\|_{L^1(\Omega)}$.
\end{enumerate}

Assume, in addition, $p\geq 2$ and that \eqref{convex:assump} holds. Then
\begin{enumerate}
\item[$(i)$] If $n\leq p+2$ then $u^\star\in L^\infty(\Omega)$. In particular, $u^\star
\in C^1_0(\overline{\Omega})$.
\item[$(ii)$] If $n>p+2$ then $u^\star\in L^\frac{np}{n-p-2}(\Omega)\cap W^{1,p}_0(\Omega)$.
\end{enumerate}
\end{thm}

\begin{rem}
If $f(u_\lambda)$ is bounded in $L^1(\Omega)$ by a constant independent of $\lambda$,
then parts $(a)$ and $(b)$ will lead automatically to the assertions $(i)$ and $(ii)$
stated in the theorem (without the requirement that $p\geq 2$ and \eqref{convex:assump} hold true).
However, as we said before, the estimate $f(u^\star)\in L^1(\Omega)$ is unknown in
the general case, \textit{i.e}, for arbitrary positive and increasing nonlinearities
$f$ satisfying \eqref{p-superlinear} and arbitrary $p>1$.
\end{rem}

\begin{open}
Is it true that $f(u^\star)\in L^1(\Omega)$ for arbitrary positive
and increasing nonlinearities $f$ satisfying \eqref{p-superlinear}?
\end{open}

Under assumptions $p\geq 2$ and \eqref{convex:assump} it is proved in \cite{S} that
$f(u^\star)\in L^r(\Omega)$ for all $1\leq r<n/(n-p')$ when $n\geq p'$ and $f(u^\star)
\in L^\infty(\Omega)$ if $n<p'$. In particular, one has $f(u^\star)\in L^1(\Omega)$
independently of the dimension $n$ and the parameter $p>1$. As a consequence,
assertions $(i)$ and $(ii)$ follow immediately from parts $(a)$ and $(b)$ of the
theorem.

To prove the $L^r$ \textit{a priori} estimates stated in part $(a)$
and $(b)$ we make three steps. First, we use the strict convexity of
the domain $\Omega$ to prove that
$$
\{x\in\Omega:{\rm dist}(x,\partial\Omega)<\varepsilon\}
\subset \{x\in\Omega:u_\lambda(x)<s\}
$$
for a suitable $s$. This is done using a moving plane procedure for
$p$-Laplace equations (see Proposition~\ref{Prop:1} below). Then,
we prove that the Morrey and Sobolev type inequalities stated in
Theorem~\ref{Theorem:Sobolev} for smooth functions, also hold for
regular solutions of \eqref{problem} when $1\leq q\leq 2$. Finally,
taking a test function $\eta$ related to ${\rm dist}(\cdot,\partial\Omega)$
in \eqref{ineq:key} and proceeding as in the proof
of Theorem~\ref{Theorem} we will obtain the $L^r$ \textit{a priori}
estimates established in the theorem.

The energy estimate established in parts (ii) and (b) of Theorem~\ref{Theorem2}
follows by extending the arguments of Nedev \cite{Nedev01} for the Laplacian
case (see also Theorem~2.9 in \cite{CS}). First, using a Poho${\rm\check{z}}$aev
identity we obtain
\begin{equation}\label{key:Poho}
\int_\Omega|\nabla u_\lambda|^p\ dx
\leq
\frac{1}{p'}\int_{\partial\Omega}|\nabla u_\lambda|^p\ x\cdot\nu\ d\sigma,
\qquad\textrm{for all }p>1\textrm{ and }\lambda\in(0,\lambda^\star),
\end{equation}
where $d\sigma$ denotes the area element in $\partial\Omega$ and
$\nu$ is the outward unit normal to $\Omega$. Then, using the strict
convexity of the domain (as in the $L^r$ estimates) and standard regularity 
estimates for $-\Delta_p u=\lambda f(u_\lambda(x))$ in a neighborhood of 
the boundary, we are able to control the right hand side of \eqref{key:Poho} by
a constant whose dependence on $\lambda$ is given by a function of 
$\|f(u_\lambda)\|_{L^1(\Omega)}$.

\begin{rem}
Let us compare our regularity results with the sharp ones proved
by Cabr\'e, Capella, and the second author in \cite{CCS09} when $\Omega$
is the unit ball $B_1$ of $\mathbb{R}^n$. In the radial case,
the extremal solution $u^\star$ of $\Pext$ is bounded if the dimension
$n< p+\frac{4p}{p-1}$. Moreover, if $n\geq p+\frac{4p}{p-1}$ then
$u^\star\in W^{1,r}_0(B_1)$ for all $1\leq r<\bar{r}_1$, where
$$
\bar{r}_1:=\frac{np}{n-2\sqrt{\frac{n-1}{p-1}}-2}.
$$
In particular, $u^\star\in L^r(B_1)$ for all $1\leq r<\bar{r}_0$, where
$$
\bar{r}_0:=\frac{np}{n-2\sqrt{\frac{n-1}{p-1}}-p-2}.
$$
It can be shown that these regularity results are sharp by taking
the exponential and power nonlinearities.

Note that the $L^r(\Omega)$-estimate established in Theorem~\ref{Theorem2}
differs with the sharp exponent $\bar{r}_0$ defined above by the term
$2\sqrt{\frac{n-1}{p-1}}$. Moreover, observe that $\bar{r}_1$ is larger
than $p$ and tends to it as $n$ goes to infinity. In particular, the best
expected regularity independent of the dimension $n$ for the extremal
solution $u^\star$ is $W^{1,p}_0(\Omega)$, which is the one we obtain
in Theorem~\ref{Theorem2}.

\end{rem}

\subsection{Outline of the paper}
The paper is organized as follows. In section \ref{section2} we
prove Theorem~\ref{thm:Ipq} and the geometric type inequalities
stated in Theorem~\ref{Theorem:Sobolev}. In section~\ref{section3}
we prove that Theorem~\ref{Theorem:Sobolev} holds for solutions of
\eqref{problem} when $1\leq q\leq 2$. Moreover we give boundary
estimates when the domain is strictly convex. In section~\ref{section5}, we
present the semi-stability condition \eqref{semi-stab1} and the
space of admissible functions $H_0$. The rest of the section deals
with the regularity of semi-stable solutions proving
Theorems~\ref{Theorem}~and~\ref{Theorem2}.

%------------------------------------------------------------------------------------
%------------------------------------------------------------------------------------
%------------------------------------------------------------------------------------
%------------------------------------------------------------------------------------
%------------------------------------------------------------------------------------
%------------------------------------------------------------------------------------
\section{Geometric Hardy-Sobolev type inequalities}\label{section2}
In this section we prove Theorems \ref{thm:Ipq} and \ref{Theorem:Sobolev}.
As we said in the introduction, the geometric inequalities established in
Theorem~\ref{Theorem:Sobolev} are new for the range $1\leq q<p$ since the case
$q\geq p$ was proved in \cite{CS}. However, we will give the proof in all
cases using Schwarz symmetrization, giving an alternative proof for the
known range of parameters $q\geq p$.

We start recalling the definition of Schwarz symmetrization of a
compact set and of a Lipschitz continuous function.
\begin{defin}\label{Schwarz-symm}
We define the \textit{Schwarz symmetrization of a compact set $D\subset\mathbb{R}^n$}
as
$$
D^*:=\left\{
\begin{array}{lll}
B_R(0) \textrm{ with } R=(|D|/|B_1|)^{1/n}&\textrm{if}&D\neq \emptyset,\\
\emptyset&\textrm{if}&D= \emptyset.
\end{array}
\right.
$$
Let $v$ be a Lipschitz continuous function in $\overline{\Omega}$ and
$\Omega_t:=\{x\in\Omega:|v(x)|\geq t\}$. We define
the \textit{Schwarz symmetrization of $v$} as
$$
v^*(x):= \sup\{t\in \mathbb{R}: x\in \Omega_t^*\}.
$$
Equivalently, we can define the Schwarz symmetrization of $v$ as
$$
v^*(x)=\inf\{t\geq 0:V(t)<|B_1||x|^n\},
$$
where $V(t):=|\Omega_t|=|\{x\in\Omega:|v(x)|> t\}|$ denotes the distribution
function of $v$.
\end{defin}

The first ingredient in the proof of Theorem \ref{thm:Ipq}
is the isoperimetric inequality for functions $v$ in
$W^{1,1}_0(\Omega)$:
\begin{equation}\label{talenti}
n|B_1|^{1/n} V(t)^{(n-1)/n}\leq P(t):=\frac{d}{dt}\int_{\{|v|\leq t\}}|\nabla v|\ dx
\qquad\text{for a.e. } t>0,
\end{equation}
where $P(t)$ stands for the perimeter in the
sense of De Giorgi (the total variation of the characteristic function of
$\{x\in \Omega: |v(x)|>t\}$).

The second ingredient is the following Sobolev inequality on compact hypersurfaces
without boundary due to Michael and Simon \cite{MS} and to Allard \cite{A}.

\begin{thm}[\cite{A,MS}]\label{ThmMS}
Let  $M\subset  \mathbb{R}^{n}$ be a $C^\infty$ immersed $(n-1)$-dimensional
compact hypersurface without boundary and $\phi\in C^\infty(M)$.
If $q\in  [1, n-1)$, then there exists a constant $A$ depending only on
$n$ and $q$ such that
\begin{equation}\label{MS}
\left( \int_M |\phi|^{q^\star} d\sigma \right)^{1/q^\star} \leq
A
\left( \int_M |\nabla \phi|^q +  |H\phi|^q \ d\sigma \right)^{1/q},
\end{equation}
where $H$ is the mean curvature of $M$, $d\sigma$ denotes the area element in $M$,
and $q^\star = \frac{(n-1)q}{n-1-q}$.
\end{thm}

As we said in the introduction it is well known that Schwarz symmetrization
preserves the $L^r$-norm and decreases the $W^{1,r}$-norm.
Let us prove that it also decreases (up to a multiplicative constant)
the functional $I_{p,q}$ defined in \eqref{Ipq} using the isoperimetric inequality
\eqref{talenti} and the geometric inequality \eqref{MS} applied to $M=M_t=\{x\in\Omega: |v(x)|=t\}$
and $\phi=|\nabla v|^{(p-1)/q}$.

%--------------------------------------------------------------------------
\begin{proof}[Proof of Theorem {\rm\ref{thm:Ipq}}]
Let $v\in C_0^\infty(\overline{\Omega})$, $p\geq 1$, and $1\leq q<n-1$.
By Sard's theorem, almost every $t\in(0,\|v\|_{L^\infty(\Omega)})$ is a
regular value of $|v|$. By definition, if $t$ is a regular value of $|v|$,
then $\left|\nabla v(x)\right|>0$ for all $x\in\Omega$ such that
$|v(x)|=t$. Therefore, $M_t:=\{x\in\Omega: |v(x)|=t\}$ is a
$C^{\infty}$ immersed $(n-1)-$dimensional compact hypersurface of
$\mathbb{R}^n$ without boundary for every regular value $t$ .
Applying inequality \eqref{MS} to $M=M_t$ and
$\phi=|\nabla v|^{(p-1)/q}$ we obtain
\begin{equation}\label{MSv}
\left( \int_{M_t} |\nabla v|^{(p-1)\frac{q^\star}{q}} \,
d\sigma \right)^{q/q^\star}
\leq
A^q \int_{M_t} \Big|\nabla_{T,v} |\nabla v|^{\frac{p-1}{q}}\Big|^q
+ |H_v|^q |\nabla v|^{p-1} \, d \sigma
\end{equation}
for a.e. $t\in(0,\|v\|_{L^\infty(\Omega)})$, where $H_v$ denotes the mean curvature of
$M_t$, $d\sigma$ is the area element in $M_t$, $A$ is the constant in \eqref{MS} which depends
only on $n$ and $q$, and
$$
q^\star:=\frac{(n-1)q}{n-1-q}.
$$
Recall that $V(t)$, being a nonincreasing function, is differentiable almost
everywhere and, thanks to the coarea formula and that almost every $t\in(0,\|v\|_{L^\infty(\Omega)})$
is a regular value of $|v|$, we have
\begin{equation*}
- V'(t) = \int_{M_t} \frac{1}{|\nabla v|} \, d\sigma
\qquad\textrm{and}\qquad
P(t) = \int_{M_t}  d\sigma\qquad\textrm{for a.e. }t\in(0,\|v\|_{L^\infty(\Omega)}).
\end{equation*}

Therefore applying Jensen inequality and then using the isoperimetric inequality
\eqref{talenti}, we obtain
\begin{equation}\label{new1}
\left( \int_{M_t} |\nabla v|^{(p-1)\frac{q^\star}{q}+1} \,
\frac{d\sigma}{|\nabla v|} \right)^{\frac{q}{q^\star}}
\geq
\frac{P(t)^{p-1+\frac{q}{q^\star}}}{\left(- V'(t)\right)^{p-1}}
\geq
\frac{(A_1 V(t)^{\frac{n-1}{n}})^{p-1+\frac{q}{q^\star}}}
{\left(- V'(t) \right)^{p-1}}
\end{equation}
for a.e. $t\in(0,\|v\|_{L^\infty(\Omega)})$, where $A_1:=n|B_1|^{1/n}$.

Note that for radial functions the inequalities in \eqref{new1} are
equalities. Therefore, since the Schwarz symmetrization $v^*$ of $v$
is a radial function and it satisfies \eqref{MSv}, with an equality
and with constant $A=|\partial B_1|^{-1/(n-1)}$, we obtain
\begin{equation}\label{new2}
\begin{array}{lll}
\displaystyle \left( \int_{\{|v^*|=t\}} |\nabla v^*|^{(p-1)\frac{q^\star}{q}} \,
d\sigma \right)^{q/q^\star}
&=&\displaystyle
|\partial B_1|^{-\frac{q}{n-1}}\int_{\{v^*=t\}} |H_{v^*}|^q |\nabla v^*|^{p-1} \, d \sigma\\
&=&\displaystyle
\frac{(A_1 V(t)^{\frac{n-1}{n}})^{p-1+\frac{q}{q^\star}}}
{\left(- V'(t) \right)^{p-1}}.
\end{array}
\end{equation}
for a.e. $t\in(0,\|v\|_{L^\infty(\Omega)})$.
Here, we used that $V(t)=|\{|v|>t\}|=|\{|v^*|>t\}|$ for a.e. $t\in(0,\|v\|_{L^\infty(\Omega)})$.

Therefore, from \eqref{MSv}, \eqref{new1}, and \eqref{new2}, we obtain
$$
|\partial B_1|^{-\frac{q}{n-1}}\int_{\{v^*=t\}} |H_{v^*}|^q |\nabla v^*|^{p-1} \, d \sigma
\leq
A^q \int_{M_t} \Big|\nabla_{T,v} |\nabla v|^{\frac{p-1}{q}}\Big|^q
+ |H_v|^q |\nabla v|^{p-1} \, d \sigma,
$$
for a.e. $t\in(0,\|v\|_{L^\infty(\Omega)})$.
Integrating the previous inequality with respect to $t$ on
\linebreak $(0,\|v\|_{L^\infty(\Omega)})$ and using the coarea formula
we obtain inequality \eqref{comp_integrals}, with the explicit constant
$C=A^\frac{q}{p} |\partial B_1|^\frac{q}{(n-1)p}$, proving the result.
\end{proof}

\begin{rem}
We obtained the explicit admissible constant
$C=A^\frac{q}{p} |\partial B_1|^\frac{q}{(n-1)p}$
in \eqref{comp_integrals}, where $A$ is the universal constant appearing
in \eqref{MS}.
\end{rem}

We prove Theorem \ref{Theorem:Sobolev} using Theorem~\ref{thm:Ipq} and known
results on one dimensional weighted Sobolev inequalities.
%--------------------------------------------------------------------------
%--------------------------------------------------------------------------
\begin{proof}[Proof of Theorem~{\rm\ref{Theorem:Sobolev}}]
Let $v\in C_0^\infty(\overline{\Omega})$ and $v^*$ its Schwarz symmetrization.
Recall that $v^*$ is defined in $B_R$ with $R=(|\Omega|/|B_1|)^{1/n}$.

(a) Assume $1+q<n<p+q$. Using H\"older inequality we obtain
\begin{equation}\label{pointwise}
\begin{array}{lll}
v^*(s)&=&\displaystyle \int_s^R |(v^*)'(\tau)|\ d\tau\\
&\leq&\displaystyle \left(\int_0^R |(v^*)'(\tau)|^p\tau^{-q}\tau^{n-1}\ d\tau\right)^{1/p}
\left(\int_s^R \tau^\frac{1+q-n}{p-1}\ d\tau\right)^{1/p'}
\end{array}
\end{equation}
for a.e. $s\in(0,R)$. In particular,
$$
v^*(s)\leq |\partial B_1|^{-1/p} \left(\frac{p-1}{p+q-n}\right)^{1/p'} \left(\frac{|\Omega|}{|B_1|}\right)^\frac{p+q-n}{np}I_{p,q}(v^*;B_R)
$$
for a.e. $s\in(0,R)$.
We conclude this case, by Theorem \ref{thm:Ipq}, noting that $\|v\|_{L^\infty(\Omega)}=v^*(0)$.

(b) Assume $n>p+q$. We use the following 1-dimensional weighted Sobolev inequality:
\begin{equation}\label{radial:Sob}
\left(\int_0^R|\varphi(s)|^{p_q^\star} s^{n-1}\ ds\right)^{1/p_q^\star}
\leq C(n,p,q)\left(\int_0^R s^{-q}|\varphi'(s)|^p s^{n-1}\ ds\right)^{1/p}
\end{equation}
with optimal constant
\begin{equation}\label{ctant:C}
C(n,p,q):=\left(\frac{p-1}{n-(p+q)}\right)^{1/p'}n^{-1/p_q^\star}
\left[\frac{\Gamma\left(\frac{np}{p+q}\right)}
{\Gamma\left(\frac{n}{p+q}\right)\Gamma\left(1+\frac{n(p-1)}{p+q}\right)}\right]^\frac{p+q}{np}
\end{equation}
stated in \cite{Trudinger97}. Applying inequality \eqref{radial:Sob} to $\varphi=v^*$
and noting that the $L^{p_q^\star}$-norm is preserved by Schwarz symmetrization, we obtain
$$
|\partial B_1|^{-1/p_q^\star}\left(\int_{\Omega}|v|^{p_q^\star}\ dx\right)^{1/p_q^\star}
\leq C(n,p,q)|\partial B_1|^{-1/p}\left(\int_{B_R} |x|^{-q}|\nabla v^*|^p \ dx\right)^{1/p}.
$$
Using Theorem \ref{thm:Ipq} again we prove \eqref{Sobolev} for $r=p_q^\star$. The remaining
cases, $1\leq r<p_q^\star$, now follow easily from H\"older inequality.

(c) Assume $n=p+q$. From \eqref{pointwise} and Theorem~\ref{thm:Ipq}
we obtain
$$
\begin{array}{lll}
v^*(s)
&\leq&\displaystyle
\left(\int_0^R |(v^*)'(\tau)|^p\tau^{-q}\tau^{n-1}\ d\tau\right)^{1/p}
\left(\int_s^R \tau^{-1}\ d\tau\right)^{1/p'}\\
&\leq&\displaystyle
|\partial B_1|^{-1/p} C I_{p,q}(v;\Omega)
\left(\ln\left(\frac{R}{s}\right)\right)^{1/p'}
\end{array}
$$
for a.e. $s\in(0,R)$. Equivalently
$$
\exp\left\{\left(\frac{v^*(s)}{|\partial B_1|^{-1/p}C I_{p,q}(v;\Omega)}\right)^{p'}\right\}
|\partial B_1|s^{n-1}
\leq
\frac{R}{s}|\partial B_1|s^{n-1}
$$
for a.e. $s\in(0,R)$. Integrating the previous inequality with respect to $s$ in $(0,R)$
we obtain
$$
\int_{B_R}\exp\left\{\left(\frac{v^*}{|\partial B_1|^{-1/p}C I_{p,q}(v;\Omega)}\right)^{p'}\right\}\ dx
\leq
|\partial B_1|\frac{R^n}{n-1}=\frac{n}{n-1}|\Omega|.
$$
We conclude the proof noting that the integral in inequality \eqref{Moser-Trudinger} is
preserved under Schwarz symmetrization.
\end{proof}

\begin{rem}\label{rmk:ctans}
Note that we obtained explicit admissible constants $C_1$, $C_2$, and $C_3$
in inequalities of Theorem \ref{Theorem:Sobolev}. More precisely,
we obtained
$$
C_1=|\partial B_1|^{-\frac{1}{p}} \left(\frac{p-1}{p+q-n}\right)^{\frac{1}{p'}}
\left(\frac{|\Omega|}{|B_1|}\right)^\frac{p+q-n}{np}
A^\frac{q}{p} |\partial B_1|^\frac{q}{(n-1)p},
$$
$$
C_2= C(n,p,q)|\partial B_1|^{\frac{1}{p_q^\star}-\frac{1}{p}}
A^\frac{q}{p} |\partial B_1|^\frac{q}{(n-1)p},
$$
and
$$
C_3=|\partial B_1|^{-\frac{1}{p}} A^{\frac{n-p}{p}}|\partial B_1|^{\frac{n-p}{(n-1)p}},
$$
where $A$ is the universal constant appearing in \eqref{MS} and $C(n,p,q)$ is
defined in \eqref{ctant:C}.

All the constants $C_i$ depend only on $n$, $p$, and $q$. However, the best constant $A$
in \eqref{MS} is unknown (even for mean convex hypersurfaces). Behind this Sobolev
inequality there is the following geometric isoperimetric inequality
\begin{equation}\label{isop:mean}
|M|^{\frac{n-2}{n-1}} \leq A_2\int_M |H(x)|\ d\sigma.
\end{equation}
Here, $M\subset  \mathbb{R}^{n}$ is a $C^\infty$ immersed $(n-1)$-dimensional
compact hypersurface without boundary and $H$ is the mean curvature of $M$ as in
Theorem \ref{ThmMS}. The best constant in \eqref{isop:mean} is also unknown
even for mean convex hypersurfaces.
\end{rem}

%-----------------------------------------------------------------------
%-----------------------------------------------------------------------
%-----------------------------------------------------------------------
%-----------------------------------------------------------------------
%-----------------------------------------------------------------------
\section{Properties of solutions of $p$-Laplace equations}\label{section3}

In this section, we first establish an \textit{a priori} $L^\infty$
estimate in a neighborhood of the boundary $\partial\Omega$ for any
regular solution $u$ of \eqref{problem} when the domain $\Omega$ is
stricly convex. More precisely, we prove that there exists positive
constants $\varepsilon$ and $\gamma$, depending only on the domain
$\Omega$, such that
\begin{equation}\label{eqqq}
\Vert u\Vert_{L^\infty(\Omega_\varepsilon)}\leq \frac{1}{\gamma} \Vert u\Vert_{L^1 (\Omega)},
\ \ \text{ where }  \Omega_\varepsilon:=\{x\in\Omega\, :\, \text{\rm dist}(x,\partial\Omega)
<\varepsilon\}.
\end{equation}
Then, we establish that the geometric inequalities of
Theorem~\ref{Theorem:Sobolev} still hold for solutions of \eqref{problem} in the
smaller range $1\leq q\leq 2$. In the next section, these two
ingredients will allow us to obtain \textit{a priori} estimates for
semi-stable solutions.

Let $u \in W_{0}^{1,p}(\Omega)$ be a weak solution (\textit{i.e.}, a solution in
the distributional sense) of the problem
\begin{equation}\label{prob}
\left\{
\begin{array}{rcll}
-\Delta_p u &=& g(u)  &\textrm{in } \Omega, \\
 u&>& 0  &\textrm{in } \Omega, \\
 u &=& 0  &\textrm{on } \partial \Omega,
\end{array}
\right.
\end{equation}
where $\Omega$ is a bounded smooth domain in $\mathbb{R}^n$, with $n
\geq 2$, and $g$ is any positive smooth nonlinearity.

We say that $u \in W_{0}^{1,p}(\Omega)$ is a \textit{regular solution} of \eqref{prob}
if it satisfies the equation in the distributional sense and $g(u)\in L^\infty(\Omega)$.
By well known regularity results for degenerate
elliptic equations, one has that every regular solution $u$ belongs to $C^{1,\alpha}
(\Omega)$ for some $\alpha\in(0,1]$ (see \cite{DB,T}). Moreover,
$u\in C^1(\overline{\Omega})$ (see \cite{Lie}). This is the best
regularity that one can hope for solutions of $p$-Laplace equations.
Therefore, equation \eqref{prob} is always meant in a distributional
sense.

We prove the boundary \textit{a priori} estimate \eqref{eqqq} through a moving
plane procedure for the $p$-Laplacian which is developed in \cite{DS}.
%--------------------------------------------------------------------------------------
%--------------------------------------------------------------------------------------
%--------------------------------------------------------------------------------------
\begin{proposition}\label{Prop:1}
Let $\Omega$ be a smooth bounded domain of $\mathbb{R}^{n}$ and $g$ any positive
smooth function. Let $u$ be any positive regular solution of \eqref{prob}.

If $\Omega$ is strictly convex, then there exist positive constants $\varepsilon$
and $\gamma$ depending only on the domain $\Omega$ such that
for every $x\in\Omega$ with $\text{\rm dist}(x,\partial\Omega)<\varepsilon$,
there exists a set $I_x\subset\Omega$ with the following properties:
$$
|I_x|\geq\gamma \qquad\text{and}\qquad
u(x) \leq u(y) \ \text{ for all }  y\in I_x.
$$
As a consequence,
\begin{equation}\label{L1boundary}
\Vert u\Vert_{L^\infty(\Omega_\varepsilon)}\leq \frac{1}{\gamma} \Vert u\Vert_{L^1 (\Omega)},
\ \ \text{ where }  \Omega_\varepsilon:=\{x\in\Omega\, :\, \text{\rm dist}(x,\partial\Omega)
<\varepsilon\}.
\end{equation}
\end{proposition}
%--------------------------------------------------------------------------------------
%--------------------------------------------------------------------------------------
%--------------------------------------------------------------------------------------
\begin{proof}
First let us observe that from the regularity of the solution
$u$ up to the boundary $\partial \Omega$ and the fact that $\Delta_p u
\leq 0$, we can apply the generalized Hopf boundary lemma \cite{V}
to see that the normal derivative $\frac{\partial u}{\partial \nu} < 0$ on
$\partial \Omega$.
Thus, if we let $Z_u := \{ x \in \Omega : \nabla u (x) = 0 \}$ be the
critical set of $u$, we have that $Z_u \cap \partial \Omega = \emptyset$.
By the compactness of both sets, there exists $\varepsilon_0 > 0$ such that
$Z_u \cap \Omega_\varepsilon = \emptyset$ for any $\varepsilon \leq \varepsilon_0$.

We will now prove that this neighborhood of the boundary is in fact
independent of the solution $u$.
In order to begin a moving plane argument we need some notations:
let $e \in S^{n-1}$ be any direction and for $\lambda \in \mathbb{R}$ let us
consider the hyperplane
$$
T = T_{\lambda,e} = \{ x \in \mathbb{R}^n : x \cdot
e = \lambda \}
$$
and the corresponding cap
$$
\Sigma = \Sigma_{\lambda,e}
= \{ x \in \Omega : x \cdot e < \lambda \}.
$$
Set
$$
a(e) = \inf_{x \in \Omega} x \cdot e
$$
and for any $x \in \Omega$, let $x' = x_{\lambda,e}$ be its reflection with respect to
the hyperplane $T$, \textit{i.e.},
$$
x' = x + (\lambda - 2 x \cdot e)\ e.
$$
For any $\lambda > a (e)$ the cap $$\Sigma' = \{ x \in \Omega :
x' \in \Sigma \}$$ is the (non-empty) reflected cap of $\Sigma$ with
respect to $T$.

Furthermore, consider the function $v(x) = u (x') = u
(x_{\lambda,e})$, which is just the reflected of $u$ with respect to
the same hyperplane.
By the boundedness of $\Omega$, for $\lambda - a(e)$ small, we have
that the corresponding reflected cap $\Sigma'$ is contained in
$\Omega$. Moreover, by the strict convexity of $\Omega$, there exists
$\lambda_0 = \lambda_0 (\Omega)$ (independent of $e$) such that $\Sigma'$
remains in $\Omega$ for any $\lambda \leq \lambda_0$.

Let us then compare the function $u$ and its reflection $v$ for such
values of $\lambda$ in the cap $\Sigma$. First of all, both functions
solve the same equation since $\Delta_p$ is invariant under
reflection; secondly, on the hyperplane $T$ the functions coincide,
whereas for any $x \in \partial \Sigma \cap \partial \Omega$ we have that $u
(x) = 0$ and that $v(x) = u (x') > 0$, since the reflection $x' \in
\Omega$.
Hence we can see that:
\begin{equation*}
\Delta_p (u) + f (u) = \Delta_p (v) + f (v) \text{ in } \Sigma, \quad
u \leq v \text{ on } \partial \Sigma.
\end{equation*}

Again by the boundedness of $\Omega$, if $\lambda - a(e)$ is small,
the measure of the cap $\Sigma$ will be small. Therefore, from the
Comparison Principle in small domains (see \cite{DS}) we have that
$u \leq v \text{ in } \Sigma$. Moreover, by Strong Comparison
Principle and Hopf Lemma, we see that $u \leq v \text{ in }
\Sigma_{\lambda,e}$ for any $a(e) < \lambda \leq \lambda_0$.
In particular, this spells that $u(x)$ is nondecreasing in the $e$
direction for all $x \in \Sigma$.

Now, fix $x_0 \in \partial \Omega$ and let $e=\nu (x_0)$ be the unit normal
to $\partial \Omega$ at $x_0$. By the convexity assumption
$T_{a(\nu(x_0)),\nu(x_0)} \cap \partial\Omega = \{ x_0 \}$.
If we let $\theta \in S^{n-1}$ be another direction close to the
outer normal $\nu (x_0)$, the reflection of the caps
$\Sigma_{\lambda,\theta}$ with respect to the hyperplane
$T_{\lambda,\theta}$ (which is close to the tangent one) would still be
contained in $\Omega$ thanks to its strict convexity.
So the above argument could be applied also to the new direction
$\theta$. In particular, we see that we can get a neighborhood
$\Theta$ of $\nu (x_0)$ in $S^{n-1}$ such that $u (x)$ is
nondecreasing in every direction $\theta \in \Theta$ and for any $x$
such that $x \cdot \theta < \frac{\lambda_0}{2}$.

By eventually taking a smaller neighborhood $\Theta$, we may assume
that $$|x \cdot (\theta - \nu (x_0))| < \lambda_0 /8$$ for any $x \in
\Sigma_{\lambda_0,\theta}$ and $\theta \in \Theta$.
Moreover, noticing that $$x \cdot \theta = x \cdot (\theta - \nu (x_0)) + x
\cdot \nu (x_0)$$ and $$\frac{\lambda_0}{2} = \frac{\lambda_0}{8} +
\frac{3 \lambda_0}{8} > x \cdot \theta > \frac{\lambda_0}{8} -
\frac{\lambda_0}{8} = 0$$ it is then easy to see that $u$ is
nondecreasing in any direction $\theta \in \Theta$ on $\Sigma_0 = \{
x \in \Omega : \frac{\lambda_0}{8} < x \cdot \nu (x_0) <\frac{3\lambda_0}{8}
\}$.

Finally, let us choose $\varepsilon = \frac{\lambda_0}{8}$. Fix any point $x \in
\Omega_\varepsilon$ and let $x_0$ be its projection onto $\partial \Omega$.
{F}rom the above arguments we see that $$u(x) \leq u( x_0 - \varepsilon \nu
(x_0)) \leq u(y)$$ for any $y \in I_x$, where $I_x \subset \Sigma_0$
is a truncated cone with vertex at $x_1$, opening angle $\Theta$, and
height $\frac{\lambda_0}{4}$.
Hence, we have obtained that there exists a positive constant $\gamma = \gamma
(\Omega, \varepsilon)$ such that $|I_x| \geq \gamma$ and $u(x) \leq u(y)$
for any $y \in I_x$.
Finally, choosing $x_\varepsilon$ as the maximum of $u$ in $\Omega_\varepsilon$, we get
\begin{equation*}
\Vert u \Vert_{L^\infty(\Omega_\varepsilon)} = u_\varepsilon (x_\varepsilon) \leq
\frac{1}{\gamma} \int_{I_{x_{\varepsilon}}} u (y) \, dy \leq
\frac{1}{\gamma} \Vert u\Vert_{L^1 (\Omega)}
\end{equation*}
which proves \eqref{L1boundary}.
\end{proof}

We will now prove that inequalities in Theorem \ref{Theorem:Sobolev}
are also valid for a positive solution $u$ of \eqref{prob} in the
smaller range $1\leq q \leq 2$.
To do this, we will construct an approximation of $u$
through smooth functions and see that, thanks to strong uniform
estimates on this approximation, we can pass to the limit in all of
the inequalities.

%--------------------------------------------------------------------------------------
%--------------------------------------------------------------------------------------
%--------------------------------------------------------------------------------------
\begin{proposition}\label{Prop:2}
Let $\Omega$ be a smooth bounded domain of $\mathbb{R}^{n}$ and $g$
any positive smooth function. Let $u$ be any positive regular
solution of \eqref{prob}. If  $1\leq q \leq 2$, then inequalities in
Theorem {\rm\ref{Theorem:Sobolev}} hold for $v=u$. Given $s>0$, the same
holds true also for $v = u-s$ and $\Omega$ replaced by $\Omega_s
:= \{ x\in \Omega : u > s\}$.
\end{proposition}
%--------------------------------------------------------------------------------------
%--------------------------------------------------------------------------------------
%--------------------------------------------------------------------------------------
\begin{proof}
Let $Z_u=\{x\in\Omega : \nabla u(x)=0\}$. Recall that by standard elliptic regularity $u \in C^\infty (\Omega
\setminus Z_u)$ and that $|Z_u| = 0$ by \cite{DS}. Therefore, $u$ is
smooth almost everywhere in $\Omega$. Let $x \in \Omega
\setminus Z_u$ and observe that for the mean curvature $H_u$ of the
level set passing through $x$ we have the following explicit
expression
\begin{equation}\label{exp1}
-(n-1) H_u = {\rm div} \left(\frac{\nabla u}{|\nabla u|}\right) =
\frac{\Delta u}{|\nabla u|} - \frac{\langle D^2 u \nabla u, \nabla u
\rangle }{|\nabla u|^3}
\end{equation}
whereas for the tangential gradient term we have
\begin{equation}\label{exp2}
\nabla_{T,u} |\nabla u| = \frac{D^2 u \nabla u}{|\nabla u|} -
\frac{\langle D^2 u \nabla u, \nabla u \rangle \nabla u }{|\nabla u|^3},
\end{equation}
where all the terms in these expressions are evaluated at $x$. Hence,
there exists a positive constant $C = C (n,p,q)$ such that
\begin{equation}\label{hsdequadro}
\left(\frac{1}{p'}|\nabla_{T,u} |\nabla u|^{\frac{p}{q}}|\right)^{q} +
|H_u|^q |\nabla u|^{p} \leq C |D^2 u|^{q} |\nabla u|^{p - q}\quad
\textrm{for a.e. }x\in\Omega.
\end{equation}
{F}rom \cite{DS} we recall the following important estimate: for any
$1 \leq q \leq 2$ there holds
\begin{equation}\label{sciunzi}
\int_\Omega |D^2 u|^{q} |\nabla u|^{p - q} \, dx < \infty.
\end{equation}
Thanks to \eqref{hsdequadro} and \eqref{sciunzi}, all of the
integrals in the geometric Hardy-Sobolev inequalities are well
defined for any $1 \leq q \leq 2$.

However, since the solution $u$ is not smooth around $Z_u$, we need to
regularize $u$ in a neighborhood of the critical set in order to apply the
inequalities of Theorem {\rm\ref{Theorem:Sobolev}}. We will now
describe an approximation argument due to Canino, Le, and Sciunzi~\cite{CLS} for
the $p(\cdot)$-Laplacian (in our case $p(x) \equiv p$ constant).

\begin{lem}[\cite{CLS}]
Let $D\subset \Omega$ be an open set, $1\leq q\leq 2$, and $\varepsilon\in(0,1)$.
Let $u\in C^1(\overline{\Omega})$ be a solution of \eqref{problem} and
$h:=g(u)$.
If $h_\varepsilon\in C^\infty(\overline{D})$ is any sequence converging to $h$
in $C^1(\overline{D})$ as $\varepsilon\downarrow 0$, then the unique solution
$v_\varepsilon$ of the following regularized problem
\begin{equation}\label{aprox:problem}
\left\{
\begin{array}{rcll}
-{\rm div} \left( (\varepsilon^2 + |\nabla v_\varepsilon|^2)^{\frac{p-2}{2}} \nabla v_\varepsilon \right)
&=& h_\varepsilon(x)  &\textrm{in } D, \\
 v_\varepsilon &=& u  &\textrm{on } \partial D.
\end{array}
\right.
\end{equation}
tends to $u$ strongly in $W^{1,p}(B)$. Moreover, there exists a constant
$C$ independent of $\varepsilon$ such that
$$
\int_D |D^2 v_\varepsilon|^{q} ( \varepsilon^2 + |\nabla
v_\varepsilon|^2)^{\frac{p-q}{2}} \, dx \leq C
$$
and
\begin{equation}\label{limve}
\lim_{\varepsilon \to 0} \int_D |D^2 v_\varepsilon|^{q} (
\varepsilon^2 + |\nabla v_\varepsilon|^2)^{\frac{p-q}{2}} \, dx =
\int_D |D^2 u|^{q} |\nabla u|^{p - q} \, dx.
\end{equation}
\end{lem}

Let $v_\varepsilon\in C^{\infty} (D)$ be the unique solution of
\eqref{aprox:problem} and let us consider a smooth cut-off function
$\eta$ with compact support contained in $\Omega$ and such that
$\eta \equiv 1$ on $D$. We can construct a smooth regularization
$u_\varepsilon$ of $u$ defining $u_\varepsilon := (1-\eta) u + \eta
v_\varepsilon$. We can then apply Theorem \ref{Theorem:Sobolev} to
any $u_\varepsilon$ to get the appropriate inequality $(a)$, $(b)$,
or $(c)$. From \cite{DB,Lie} and standard elliptic regularity we
know that the regularization $u_\varepsilon$ will converge to $u$,
as $\varepsilon \downarrow 0$, both in $C^1 (\overline{\Omega})$ and
$C^2(\overline{\Omega} \setminus Z_u)$. Hence we can easily pass to
the limit as $\varepsilon \downarrow 0$ in the left hand side of
\eqref{Morrey} and \eqref{Sobolev}.

In order to see that also the remaining terms $I_{p,q}
(u_\varepsilon; \Omega)$ which involve tangential gradient and mean
curvature behave well under this approximation the argument is the
following. Splitting the domain $\Omega$ and recalling that
$u_\varepsilon \equiv v_\varepsilon$ in $D$ we have that: $$I_{p,q}
(u_\varepsilon;\Omega) = I_{p,q} (u_\varepsilon; D) + I_{p,q}
(u_\varepsilon; \Omega \setminus D) = I_{p,q} (v_\varepsilon; D) +
I_{p,q} (u_\varepsilon; \Omega \setminus D).$$ Clearly, from the
$C^2$ convergence we have that $I_{p,q} (u_\varepsilon; \Omega
\setminus D) \to I_{p,q} (u; \Omega \setminus D)$ as $\varepsilon
\downarrow 0$. Therefore we can concentrate on the convergence of
$I_{p,q}(v_\varepsilon;D)$.

{F}rom \eqref{exp1}, \eqref{exp2}, and through a simple expansion of
$(\varepsilon^2 + |\nabla v_\varepsilon|^2)^{\frac{p - q}{2}}$
around $\varepsilon =0$, we see that for a sufficiently small
$\varepsilon_0 > 0$ there exists a constant $K = K (n,p,q,
\varepsilon_0) > 0$ such that for any $\varepsilon \leq
\varepsilon_0$ we have

\begin{equation}\label{hsdequadrove}
\left(\frac{1}{p'} |\nabla_{T,v_\varepsilon}|\nabla v_\varepsilon|^{\frac{p}{q}}|
\right)^{q} + |H_{v_\varepsilon}|^q |\nabla v_\varepsilon|^{p} \leq
K \, |D^2 v_\varepsilon|^{q} (\varepsilon^2 + |\nabla
v_\varepsilon|^2)^{\frac{p - q}{2}}.
\end{equation}

Moreover, by the fact that $v_\varepsilon \to u$ in $C^2 (D
\setminus Z_u)$ and $|Z_u|=0$, almost everywhere in $D$ we have

\begin{equation}\label{aeve}
\lim_{\varepsilon \to 0}\left(\frac{1}{p'} |\nabla_{T,v_\varepsilon} |\nabla
v_\varepsilon|^{\frac{p}{q}}| \right)^{q} + |H_{v_\varepsilon}|^q
|\nabla v_\varepsilon|^{p}  = \left(\frac{1}{p'} |\nabla_{T,u} |\nabla
u|^{\frac{p}{q}}| \right)^{q} + |H_{u}|^q |\nabla u|^{p}.
\end{equation}

Now, thanks to \eqref{limve}, \eqref{hsdequadrove},
and \eqref{aeve}, by dominated convergence theorem we see that:
$$
\lim_{\varepsilon \to 0} \int_D \left(\frac{1}{p'} |\nabla_{T,v_\varepsilon} |\nabla
v_\varepsilon|^{\frac{p}{q}}| \right)^{q} + |H_{v_\varepsilon}|^q
|\nabla v_\varepsilon|^{p} \, dx 
$$
$$
= \int_D \left(\frac{1}{p'} |\nabla_{T,u} |\nabla
u|^{\frac{p}{q}}| \right)^{q} + |H_{u}|^q |\nabla u|^{p} \, dx.
$$
Thus, the assertions of Theorem~\ref{Theorem:Sobolev} hold for
$v=u$.

To conclude the proof let us fix any $s >0$ and consider $v = u-s$
on $\Omega_s = \{x\in\Omega : u > s \}$. It is clear that the integrands in the
inequalities remain unchanged in this case, so the only problem
comes from the fact $\Omega_s$ might not be smooth. If this is the
case, let us consider two sequences $\varepsilon_n \to 0$ and $s_n
\to s$, with the corresponding regularizations of $v$ given by $v_n
:= v_{\varepsilon_n} = u_{\varepsilon_n} - s_n$. Thanks to the
smoothness of any $v_n$ and Sard Lemma, we can choose each $s_n$ as
a regular value of $v_n$, so that the level set $\{ v_n > 0 \} = \{
u_n > s_n \}$ is smooth. Moreover, from the $C^1$ convergence, it is
clear that for the characteristic functions we have $\chi_{\{ u_n >
s_n \}} \to \chi_{\{ u > s \}}$. Hence we can conclude the proof
using the same dominated convergence argument as above.\end{proof}

%-----------------------------------------------------------------------
%-----------------------------------------------------------------------
%-----------------------------------------------------------------------
%-----------------------------------------------------------------------
%-----------------------------------------------------------------------
\section{Regularity of stable solutions. Proof of Theorems \ref{Theorem}
and \ref{Theorem2}}\label{section5}

We are now ready to establish $L^r$ and $W^{1,r}$ \textit{a priori}
estimates of semi-stable solutions to $p$-Laplace equations proving
Theorems~\ref{Theorem}~ and~\ref{Theorem2}.

Before the proof our regularity results let us recall some known
facts on the linearized operator associated to \eqref{problem} and
semi-stable solutions.

%-----------------------------------------------------------------------
%-----------------------------------------------------------------------
%-----------------------------------------------------------------------
%-----------------------------------------------------------------------
%-----------------------------------------------------------------------
\subsection{Linearized operator and semi-stable solutions}\label{section4}

This subsection deals with the linearized operator at any regular semi-stable
solution $u \in C_{0}^{1} (\overline{\Omega})$ of
\begin{equation}\label{prob2}
\left\{
\begin{array}{rcll}
-\Delta_p u &=& g(u)  &\textrm{in } \Omega, \\
 u&>& 0  &\textrm{in } \Omega, \\
 u &=& 0  &\textrm{on } \partial \Omega,
\end{array}
\right.
\end{equation}
where $\Omega$ is a bounded smooth domain in $\mathbb{R}^n$,
with $n \geq 2$, and $g$ is any positive $C^1$ nonlinearity.

The linearized operator $L_u$ associated to \eqref{prob2} at $u$ is defined
by duality as
$$
\begin{array}{l}
L_u(v,\phi):=\displaystyle \hspace{-0.2cm}\int_\Omega |\nabla u|^{p-2}\left\{\nabla v\cdot\nabla\phi
+(p-2)\left(\nabla v\cdot\frac{\nabla u}{|\nabla u|}\right)
\left(\nabla\phi\cdot\frac{\nabla u}{|\nabla u|}\right)\right\} dx\\
\displaystyle \hspace{2cm}- \int_\Omega g'(u)v \phi\ dx
\end{array}
$$
for all $(v,\phi)\in H_0\times H_0$, where the Hilbert space $H_0$ is defined
according to \cite{DS} as follows.

\begin{defin}\label{H0}
Let $u \in C_{0}^{1} (\overline{\Omega})$ be a regular semi-stable solution
of \eqref{prob2}.
We introduce the following weighted $L^2$-norm of the gradient
$$
|\phi|:=\left(\int_\Omega \rho |\nabla \phi|^2\ dx\right)^{1/2}
\quad \textrm{where }\rho:=|\nabla u|^{p-2}.
$$
According to \cite{DS}, the space
$$
H^1_\rho(\Omega):=\{\phi \in L^2(\Omega) \hbox{ weakly differentiable}:\:|\phi|<+\infty\}
$$
is a Hilbert space and is the completion of $C^\infty(\Omega)$ with respect to the
$|\cdot|$-norm.

We define the Hilbert space $H_0$ of admissible test functions as
$$
H_0:=
\left\{
\begin{array}{lll}
\{\phi \in H_0^1(\Omega):\: |\phi|<+\infty \}&\textrm{if}&1<p\leq 2\\
\\
\textrm{the closure of }C_0^\infty(\Omega) \textrm{ in }H^1_\rho(\Omega)&\textrm{if}&p>2.
\end{array}
\right.
$$
\end{defin}

Note that for $1<p\leq 2$, $H_0$ is a subspace of $H_0^1(\Omega)$ and since
$$
\int_\Omega |\nabla \phi|^2 \leq \|\nabla u\|_{L^\infty(\Omega)}^{2-p}|\phi|^2,
$$
we see that $(H_0,|\cdot|)$ is a Hilbert space. For $p>2$, the
weight $\rho=|\nabla u|^{p-2}$ is in $L^\infty(\Omega)$ and
satisfies $\rho^{-1} \in L^1(\Omega)$, as shown in \cite{DS}.

Now, thanks to the above definition, the operator $L_u$ is well defined
for $\phi \in H_0$ and, therefore, the semistability of the solution $u$ reads as
\begin{equation}\label{semistab}
L_u (\phi,\phi) =  \int_\Omega |\nabla u|^{p-2} \left\{|\nabla \phi|^2
+(p-2)\left(\nabla \phi\cdot\frac{\nabla u}{|\nabla u|}\right)^2\right\}
- g'(u) \phi^2\ dx \geq 0,
\end{equation}
for every $\phi \in H_0$.

On the one hand, considering $\phi = |\nabla u| \eta$ as a test function
in the semistability condition \eqref{semistab} for $u$, we obtain
\begin{equation}\label{StZu}
\int_{\Omega}\left[ (p-1) |\nabla u|^{p-2} |\nabla_{T,u} |\nabla u||^{2}
+ B_u^2 |\nabla u|^{p} \right] \eta^2 \, dx
\leq (p-1) \int_{\Omega}  |\nabla u|^{p} |\nabla \eta|^2 \, dx
\end{equation}
for any Lipschitz continuous function $\eta$ with compact support.
Here, $B_u^2$ denotes the $L^2$-norm of the second fundamental form
of the level set of $|u|$ through $x$ (\textit{i.e.}, the sum of the
squares of its principal curvatures). The fact that $\phi=\eta
|\nabla u|$ is an admissible test function derives from the estimate
\eqref{sciunzi}, whereas the computations behind \eqref{StZu} are
done in \cite{FSV} (see Theorem~2.5 \cite{FSV}).

On the other hand, noting that $(n-1) H_u^{2} \leq B_u^2$ and
$$
|\nabla u|^{p-2} |\nabla_{T,u} |\nabla u||^{2} = \frac{4}{p^2} |\nabla_{T,u} |\nabla u|^{\frac{p}{2}}|^{2},
$$
we obtain the key inequality to prove our regularity results for
semi-stable solutions
\begin{equation}\label{StZu2}
\int_{\Omega}\left( \frac{4}{p^2}|\nabla_{T,u} |\nabla u|^{p/2}|^{2}
+ \frac{n-1}{p-1}H_u^2 |\nabla u|^{p} \right) \eta^2 \, dx
\leq \int_{\Omega}  |\nabla u|^{p} |\nabla \eta|^2 \, dx
\end{equation}
for any Lipschitz continuous function $\eta$ with compact support.

%-----------------------------------------------------------------------
%-----------------------------------------------------------------------
\subsection{\textit{A priori} estimates of stable solutions.
Proof of Theorem~\ref{Theorem}}\label{subsection5:1}
%-----------------------------------------------------------------------
%-----------------------------------------------------------------------

In order to prove the gradient estimate \eqref{grad:estimate}
established in Theorem~{\rm \ref{Theorem}}~(b) we will use the
following result. Its proof is based on a technique introduced by
B\'enilan \textit{et al.} \cite{BBGGPV95} to obtain the regularity
of entropy solutions for $p$-Laplace equations with $L^1$ data.

%-------------------------------------------------------------------------------
%-------------------------------------------------------------------------------
%-------------------------------------------------------------------------------
\begin{proposition}\label{Prop:bootstrap}
Assume $n\geq 3$ and $h\in L^1(\Omega)$. Let $u$ be the entropy solution of
\begin{equation}\label{linear}
\left\{
\begin{array}{rcll}
-\Delta_p u&=&h(x)&\textrm{in }\Omega,\\
u&=&0&\textrm{on }\partial \Omega.
\end{array}
\right.
\end{equation}
Let $r_0\geq (p-1)n/(n-p)$. If $\int_\Omega |u|^{r_0}\ dx<+\infty$, then the
following \textit{a priori} estimate holds:
$$
\int_\Omega |\nabla u|^r\ dx
\leq
r|\Omega|
+\left(\frac{r_1}{r}-1\right)^{-1}\left(\int_\Omega |u|^{r_0}\ dx+\|h\|_{L^1(\Omega)}\right)
$$
for all $r< r_1:=pr_0/(r_0+1)$.
\end{proposition}
%-------------------------------------------------------------------------------
%-------------------------------------------------------------------------------
%-------------------------------------------------------------------------------
\begin{rem}
B\'enilan \textit{et al.} \cite{BBGGPV95} proved the existence and
uniqueness of entropy solutions to problem \eqref{linear}. Moreover,
they proved that $|\nabla u|^{p-1}\in L^r(\Omega)$ for all $1\leq
r<n/(n-1)$ and $|u|^{p-1}\in L^r(\Omega)$ for all $1\leq r<n/(n-p)$.
Proposition~\ref{Prop:bootstrap} establishes an improvement of the
previous gradient estimate knowing an \textit{a priori} estimate of
$\int_\Omega |u|^{r_0} dx$ for some $r_0>(p-1)n/(n-p)$.
\end{rem}

\begin{proof}[Proof of Proposition {\rm\ref{Prop:bootstrap}}]
Multiplying \eqref{linear} by $T_s u=\max\{-s,$ $\min\{s,u\}\}$ we obtain
$$
\int_{\{|u|\leq s\}}|\nabla u|^p\ dx=\int_\Omega h(x)T_su\ dx\leq s\|h\|_{L^1(\Omega)}.
$$
Let $t=s^{(r_0+1)/p}$. {F}rom the previous inequality, recalling that
$V(s)=|\{x\in\Omega:|u|>s\}|$, we deduce
$$
\begin{array}{lll}
\displaystyle s^{r_0}|\{|\nabla u|>t\}|&\leq&\displaystyle
s^{r_0}\int_{\{|\nabla u|>t\}\cap\{|u|\leq s\}}\left(\frac{|\nabla u|}{t}\right)^p dx
+s^{r_0}\int_{\{|u|>s\}}\ dx\\
\\
&\leq &\displaystyle\|h\|_{L^1(\Omega)} +s^{r_0}V(s)\quad\textrm{for a.e. }s>0.
\end{array}
$$
In particular
\begin{equation}\label{lalarito}
t^{\frac{pr_0}{r_0+1}}|\{|\nabla u|>t\}|
\leq
\|h\|_{L^1(\Omega)}+\sup_{\tau>0}\Big\{\tau^{r_0}V(\tau)\Big\}\
\quad\textrm{for a.e. }t>0.
\end{equation}

Moreover, since
$$
\tau^{r_0}V(\tau)
\leq
\tau^{r_0}\int_{\{|u|>\tau\}}\left(\frac{|u|}{\tau}\right)^{r_0}\ dx
\leq\int_\Omega|u|^{r_0}\ dx \quad\textrm{for a.e. }\tau>0,
$$
we have $\sup_{\tau>0}\Big\{\tau^{r_0}V(\tau)\Big\}\leq \int_\Omega|u|^{r_0}\ dx$.

Let $r< r_1:=pr_0/(r_0+1)$. From \eqref{lalarito} and the previous inequality,
we have
$$
\begin{array}{lll}
\displaystyle \int_\Omega |\nabla u|^r\ dx
&=&
\displaystyle r\int_0^\infty t^{r-1}|\{|\nabla u|>t\}|\ dt\\
&\leq&
\displaystyle
r|\Omega|+r\left(\int_\Omega|u|^{r_0}\ dx
+
\|h\|_{L^1(\Omega)}\right)
\int_1^\infty t^{r-1}t^{-\frac{p{r_0}}{{r_0}+1}}\ dt
\end{array}
$$
proving the proposition.
\end{proof}

Now, we have all the ingredients to prove the \textit{a priori} estimates
established in Theorem~\ref{Theorem} for semi-stable solutions. It will follow
from Theorem~\ref{Theorem:Sobolev} and Propositions~\ref{Prop:2} and \ref{Prop:bootstrap}
choosing adequate test functions in the semistability condition \eqref{StZu2}.

First, we prove Theorem \ref{Theorem} when $n\neq p+2$. We will take $\eta=T_s u
=\min\{s,u\}$ as a test function in \eqref{StZu2} and then, thanks to
Proposition~\ref{Prop:2}, we apply our Morrey and Sobolev inequalities
\eqref{Morrey} and \eqref{Sobolev} with $q=2$.

%-------------------------------------------------------------------------------
%-------------------------------------------------------------------------------
%-------------------------------------------------------------------------------
\begin{proof}[Proof of Theorem {\rm \ref{Theorem}} for $n\neq p+2$]
Assume $n\neq p+2$. Let $u\in C^1_0(\overline{\Omega})$ be a semi-stable solution
of \eqref{problem}. By taking $\eta=T_s u=\min\{s,u\}$ in the semistability condition
\eqref{StZu2} we obtain
$$
\int_{\{u>s\}}\left( \frac{4}{p^2}|\nabla_{T,u} |\nabla u|^{p/2}|^{2}
+ \frac{n-1}{p-1}H_u^2 |\nabla u|^{p} \right)\, dx
\leq \frac{1}{s^2}\int_{\{u<s\}}  |\nabla u|^{p+2}\, dx
$$
for a.e. $s>0$. In particular,
$$
\min\left(\frac{4}{(n-1)p},1\right) I_{p,2}(u-s;\{x\in\Omega:u>s\})^p
\leq
\frac{p-1}{(n-1)s^2}\int_{\{u<s\}}  |\nabla u|^{p+2}\, dx
$$
for a.e. $s>0$, where $I_{p,2}$ is the functional defined in
\eqref{Ipq} with $q=2$. By Proposition~\ref{Prop:2} we can apply
Theorem~\ref{Theorem:Sobolev} with $\Omega$ replaced by
$\{x\in\Omega: u > s\}$, $v=u-s$, and $q=2$. Then, the $L^r$
estimates established in parts (a) and (b) follow directly from the
Morrey and Sobolev type inequalities \eqref{Morrey} and
\eqref{Sobolev}.

Finally, the gradient estimate \eqref{grad:estimate} follows directly from
Proposition~\ref{Prop:bootstrap} with $r_0=np/(n-p-2)$.
\end{proof}

Now, we deal with the proof of Theorem {\rm \ref{Theorem}} $($a$)$ when $n=p+2$.
This critical case follows from Theorem \ref{ThmMS} and the semistability condition
\eqref{StZu2} with the test function $\eta=\eta(u)$ defined in \eqref{etaa} and
\eqref{psii} below.

\begin{proof}[Proof of Theorem {\rm \ref{Theorem}} when $n=p+2$]
Assume $n = p+2$ (and hence, $n>3$). Taking a Lipschitz function
$\eta = \eta (u)$ (to be chosen later) in \eqref{StZu} and using the coarea
formula we obtain
\begin{equation}\label{semi:n=p+2}
\begin{array}{l}
C\displaystyle \int_{0}^{\infty}  \int_{\{ u = t \}}
\left\{\left|\nabla_{T,u} |\nabla u|^\frac{p-1}{2}\right|^{2}
+ \left|H_u|\nabla u|^\frac{p-1}{2}\right|^2\right\}\ \eta(t)^2 \, d\sigma dt
\\
\displaystyle \hspace{3cm}
\leq
\int_{0}^{\infty}  \int_{\{ u = t \}}  |\nabla u|^{p+1}\ \dot{\eta}(t)^2  \, d\sigma dt,
\end{array}
\end{equation}
where $d\sigma$ denotes the area element in $\{u=t\}$ and $C$, here and in the rest
of the proof, is a constant depending only on $p$.

To apply the Sobolev inequality \eqref{MS} in the left hand side of the
previous inequality we need to make an approximation argument.
Consider the sequence $u_k$ of smooth regularizations of $u$ introduced in
the proof of Proposition \ref{Prop:2} and note that $\{u_k=t\}$ is a smooth
hypersurface for a.e. $t\geq0$.
Then, from the Sobolev inequality \eqref{MS} with $\phi = |\nabla u_k|^{\frac{p-1}{2}}$,
$q=2$, and $M=\{u_k=t\}$, and noting that
$$
(p-1) \frac{n-1}{n-3} = p+1\quad \textrm{ when }n = p+2,
$$
we obtain
\begin{equation}\label{semi:u_k}
\begin{array}{l}
\displaystyle
C \int_{0}^{\infty} \left(\int_{\{ u_k = t \}}
|\nabla u_k|^{p+1}\right)^\frac{n-3}{n-1}  \eta(t)^2 \, d\sigma \ dt
\\
\hspace{1.0cm} \leq
\displaystyle \int_{0}^{\infty} \int_{\{ u_k = t \}}
\left\{\left|\nabla_{T,u_k} |\nabla u_k|^\frac{p-1}{2}\right|^{2}
+ \left|H_{u_k}|\nabla u_k|^\frac{p-1}{2}\right|^2\right\}\ \eta(t)^2 \, d\sigma dt.
\end{array}
\end{equation}

Now, we will pass to the limit in the previous inequality. Note that,
if $\eta$ is bounded, through a dominated convergence argument as in
Proposition \ref{Prop:2} we have
$$
\begin{array}{l}
\displaystyle \lim_{k \to \infty} \int_{0}^{\infty} \int_{\{ u_k = t \}}
\left\{\left|\nabla_{T,u_k} |\nabla u_k|^\frac{p-1}{2}\right|^{2}
+ \left|H_{u_k}|\nabla u_k|^\frac{p-1}{2}\right|^2\right\}\ \eta(t)^2 \, d\sigma dt
\\
\displaystyle
\hspace{0.5cm}
=\int_{0}^{\infty} \int_{\{ u = t \}}
\left\{\left|\nabla_{T,u} |\nabla u|^\frac{p-1}{2}\right|^{2}
+ \left|H_u|\nabla u|^\frac{p-1}{2}\right|^2\right\}\ \eta(t)^2 \, d\sigma dt.
\end{array}
$$
Moreover, from the $C^1$ convergence of $u_k$ to $u$ we obtain
$$
\lim_{k \to \infty} \int_{0}^{\infty} \left(\int_{\{ u_k = t \}}
|\nabla u_k|^{p+1}\right)^\frac{n-3}{n-1}  \eta(t)^2\, d\sigma \ dt
=
\int_{0}^{\infty}\left(\int_{\{ u = t \}}
|\nabla u|^{p+1}\right)^\frac{n-3}{n-1}  \eta(t)^2\, d\sigma \ dt.
$$

Therefore, taking the limit as $k$ goes to infinity in \eqref{semi:u_k} and using
\eqref{semi:n=p+2}, we get
\begin{equation}\label{p+2}
C \int_{0}^{\infty}  \psi(t)^{\frac{n-3}{n-1}} \, \eta (t)^2 \, dt
\leq \int_{0}^{\infty} \psi(t) \, \dot{\eta} (t)^2 \, dt =
\int_{0}^{\infty} \int_{\{ u = t \}}  |\nabla u|^{p+1}   \,
d\sigma\, \dot{\eta} (t)^{2}\, dt,
\end{equation}
where
\begin{equation}\label{psii}
\psi(t) := \int_{\{ u = t \}}  |\nabla u|^{p+1} \, d\sigma.
\end{equation}

Now, let $\bar{M}:={\| u \|_{L^\infty(\Omega)}}$. Given $s > 0$, choose
\begin{equation}\label{etaa}
\eta(t)=\eta_s (t) :=\left\{
\begin{array}{lll}
\displaystyle t/s&\textrm{ if }& 0\leq t \leq s,\\
\displaystyle \exp \left( \frac{1}{\sqrt{2}} \int_{s}^{t}
\left(\frac{C  \psi(\tau)^{\frac{n-3}{n-1}}}{ \psi(\tau)}\right)^{\frac12}
\, d\tau \right)&\text{ if }& s < t \leq \bar{M}\\
\displaystyle \eta(M)&\textrm{ if }& t >\bar{M}.
\end{array}
\right.
\end{equation}
It is then clear that

\begin{equation*}
\int_{0}^{\bar{M}} \int_{\{ u = t \}}  |\nabla u|^{p+1}   \, d\sigma \, \dot{\eta}_{s} (t)^{2}\, dt
=
\frac{1}{s^2} \int_{\{ u \leq s \}}  |\nabla u|^{p+2} \, dx
+ \frac{C}{2} \int_{s}^{\bar{M}} \psi(t)^{\frac{n-3}{n-1}} \, \eta_{s} (t)^{2} \, dt.
\end{equation*}
Therefore, from \eqref{p+2} we obtain
\begin{equation}\label{p+2:bis}
\frac{C}{2} \int_{s}^{\bar{M}} \psi(t)^{\frac{n-3}{n-1}} \, \eta_{s} (t)^{2} \, dt
\leq
\frac{1}{s^2} \int_{\{ u \leq s \}}  |\nabla u|^{p+2} \, dx.
\end{equation}

Let us choose $\alpha = \frac{2}{n-2}$, $\beta = \frac{n-3}{(n-2)(n-1)}$, and $m = n-2$.
Note that $\alpha,\beta>0$, $m>1$, and $\beta m'=1/(n-1)$. Moreover, using the definition
of $\eta_s$ we have
\begin{equation}\label{asdf}
\frac{1}{\psi(t)^{\beta m'} \eta_{s}(t)^{\alpha m'}}
=
\sqrt{\frac{2}{C}}\frac{\dot{\eta}_{s}(t)}{\eta_{s}(t)^{\alpha m' + 1}}
\end{equation}
for all $t>s$. By \eqref{asdf}, H\"older inequality, and
\eqref{p+2:bis}, we see that
$$
\begin{array}{lll}
\displaystyle \bar{M} - s
&=&\displaystyle
\int_{s}^{\bar{M}} \frac{\psi(t)^{\beta}\eta_{s}(t)^{\alpha}}{\psi(t)^{\beta}\eta_{s}(t)^{\alpha}} \, dt\\
&\leq&\displaystyle
\left( \int_{s}^{\bar{M}} \psi(t)^{\beta m}\eta_{s}(t)^{\alpha m} \, dt \right)^{\frac{1}{m}}
\left( \int_{s}^{\bar{M}} \frac{dt}{\psi(t)^{\beta m'}\eta_{s}(t)^{\alpha m'}} \right)^{\frac{1}{m'}}\\
&\leq&\displaystyle
\left( \int_{s}^{\bar{M}} \psi(t)^{\frac{n-3}{n-1}}\eta_{s}(t)^{2} \, dt \right)^{\frac{1}{n-2}}
\left( \sqrt{\frac{2}{C}}\int_{s}^{\bar{M}} \frac{\dot{\eta}_{s} (t)}
{\eta_{s}(t)^{m'\alpha + 1}}  \, dt \right)^{\frac{n-3}{n-2}}\\
&\leq&\displaystyle
\left( \frac{2}{Cs^2} \int_{\{ u\leq s \}}  |\nabla u|^{p+2} \, dx \right)^{\frac{1}{n-2}}
\left(\sqrt{\frac{2}{C}}\frac{n-3}{2}\right)^{\frac{n-3}{n-2}}
\end{array}
$$
which is exactly \eqref{L-infinty} (note that $n-2=p$ and $\eta(\bar{M})\geq 1$).
\end{proof}

%-----------------------------------------------------------------------
%-----------------------------------------------------------------------
%-----------------------------------------------------------------------
\subsection{Regularity of the extremal solution.
Proof of Theorem~\ref{Theorem2}}\label{subsection5:2}
%-----------------------------------------------------------------------
%-----------------------------------------------------------------------
%-----------------------------------------------------------------------

In this subsection we will prove the \textit{a priori} estimates for minimal
and extremal solutions of $\Plambda$ stated in Theorem~\ref{Theorem2}.
Let us remark that in the proof of Theorem~\ref{Theorem2} we will assume
the nonlinearity $f$ to be smooth. However, if it is only $C^1$ we can
proceed with an approximation argument as in the proof of Theorem 1.2 in
\cite{Cabre09}.

The $W^{1,p}$-estimate established in Theorem~\ref{Theorem2} has as main
ingredient the following result.

%---------------------------------------------------------------------
%---------------------------------------------------------------------
\begin{lem}\label{lemma:Poho}
Let $f$ be an increasing positive $C^1$ function satisfying \eqref{p-superlinear}
and $\lambda\in(0,\lambda^\star)$.
Let $u=u_\lambda\in C^1_0(\overline{\Omega})$ be the minimal solution of $\Plambda$.
The following inequality holds:
\begin{equation}\label{Pohozaev}
\int_\Omega|\nabla u|^p\ dx
\leq
\left(\max_{x\in\overline{\Omega}}|x|\right)
\frac{1}{p'}\int_{\partial\Omega}|\nabla u|^p\ d\sigma.
\end{equation}
\end{lem}
%---------------------------------------------------------------------
%---------------------------------------------------------------------
\begin{proof}
Let $G'(t)=g(t)=\lambda f(t)$. First, we note that
$$
x\cdot\nabla u\ g(u) =x\cdot\nabla G(u)={\rm
div}\Big(G(u)x\Big)-nG(u)
$$
and that almost everywhere on $\Omega$ we can evaluate
$$
\begin{array}{lll}
\displaystyle x\cdot\nabla u\ \Delta_p u
-{\rm div}\Big(x\cdot\nabla u\ |\nabla u|^{p-2}\nabla u\Big)
%\hspace{2.5cm}
&=&\displaystyle
-|\nabla u|^{p-2}\nabla u\cdot\nabla(x\cdot\nabla u)\\
&=&
\displaystyle
 -|\nabla u|^{p}-\frac{1}{p}\nabla|\nabla u|^p\cdot x\\
&=&
\displaystyle
 \frac{n-p}{p}|\nabla u|^{p}-\frac{1}{p}{\rm div} \Big(|\nabla u|^p x\Big).
\end{array}
$$
As a consequence, multiplying $\Plambda$ by $x\cdot\nabla u$ and integrating on $\Omega$,
we have
\begin{equation}\label{Pohozaev0}
n\int_\Omega G(u)\ dx-\frac{n-p}{p}\int_\Omega |\nabla u|^p\ dx=
\frac{1}{p'}\int_{\partial\Omega}|\nabla u|^p\ x\cdot\nu\ d\sigma,
\end{equation}
where $\nu$ is the outward unit normal to $\Omega$.

Noting that $u$ is an absolute minimizer of the energy functional
$$
J(u)=\frac{1}{p}\int_\Omega|\nabla u|^p\ dx-\int_\Omega G(u)\ dx
$$
in the convex set $\{v\in W^{1,p}_0(\Omega):0\leq v\leq u\}$ (see \cite{CS07}),
we have that $J(u)\leq J(0)=0$. Therefore, from \eqref{Pohozaev0}
we obtain
$$
\int_\Omega|\nabla u|^p\ dx
=
n J(u)
+\frac{1}{p'}\int_{\partial\Omega}|\nabla u|^p\ x\cdot\nu\ d\sigma
\leq
\left(\max_{x\in\overline{\Omega}}|x|\right)
\frac{1}{p'}\int_{\partial\Omega}|\nabla u|^p\ d\sigma
$$
proving the lemma.
\end{proof}

Finally, we prove Theorem~\ref{Theorem2} (using the semistability
condition \eqref{StZu2} with an appropriate test function),
Theorem~\ref{Theorem:Sobolev}, and Lemma~\ref{lemma:Poho}.

\begin{proof}[Proof of Theorem {\rm \ref{Theorem2}}]
Let $u_\lambda$ be the minimal solution of $\Plambda$ for $\lambda\in(0,\lambda^\star)$.
{F}rom \cite{CS07} we know that minimal solutions are semi-stable. In particular, $u_\lambda$
satisfies the semistability condition \eqref{StZu2} for all $\lambda\in(0,\lambda^\star)$.

Assume that $\Omega$ is strictly convex. Let $\delta(x) := {\rm
dist}(x,\partial\Omega)$ be the distance to the boundary and
$\Omega_\varepsilon:=\{x\in\Omega:\delta (x)<\varepsilon\}$. By
Proposition~\ref{Prop:1} there exist positive constants
$\varepsilon$ and $\gamma$ such that for every $x_0\in
\Omega_\varepsilon$ there exists a set $I_{x_0}\subset \Omega$
satisfying $|I_{x_0}|>\gamma$ and
\begin{equation}\label{kkkkey}
u_\lambda(x_0)^{p-1}
\leq
u_\lambda(y)^{p-1}\quad\textrm{for all }y\in I_{x_0}.
\end{equation}
Let $x_\varepsilon\in\overline{\Omega}_\varepsilon$ be such that
$u_\lambda(x_\varepsilon)=\|u_\lambda\|_{L^\infty(\Omega_\varepsilon)}$.
Integrating with respect to $y$ in $I_{x_\varepsilon}$ inequality
\eqref{kkkkey} and using \eqref{p-superlinear}, we obtain
\begin{equation}\label{kkkkeyyyyy}
\|u_\lambda\|_{L^\infty(\Omega_\varepsilon)}^{p-1}
\leq
\frac{1}{\gamma}\int_{I_{x_\varepsilon}}u_\lambda^{p-1}\ dy
\leq
\frac{1}{\gamma}\int_{\Omega}u_\lambda^{p-1}\ dy
\leq \frac{C}{\gamma}\|f(u_\lambda)\|_{L^1(\Omega)},
\end{equation}
where $C$, here and in the rest of the proof, is a constant independent of $\lambda$. Letting
$s=\left(\frac{C}{\gamma}\|f(u_\lambda)\|_{L^1(\Omega)}\right)^{1/(p-1)}$,
we deduce
\begin{equation}\label{ghjklk}
\Omega_\varepsilon\subset\{x\in\Omega:u_\lambda(x) \leq s\}.
\end{equation}

Now, choose
$$
\eta (x) := \left\{
\begin{array}{lll}
\delta (x)&\textrm{if}&\delta (x) < \varepsilon,\\
\varepsilon&\textrm{if}&\delta (x) \geq \varepsilon,
\end{array}
\right.
$$
as a test function in \eqref{StZu2} and use \eqref{ghjklk} to obtain
$$
\varepsilon^2
\int_{\{u_\lambda>s\}}\left( \frac{4}{p^2}|\nabla_{T,u_\lambda} |\nabla u_\lambda|^{p/2}|^{2}
+ \frac{n-1}{p-1}H_{u_\lambda}^2 |\nabla u_\lambda|^{p} \right) \, dx
\leq \int_{\{u_\lambda \leq s\}}  |\nabla u_\lambda|^{p} \, dx.
$$
Multiplying equation $\Plambda$ by $T_su_\lambda=\min\{s,u_\lambda\}$ we have
\begin{equation}\label{umens}
\int_{\{u_\lambda<s\}}|\nabla u_\lambda|^p\ dx=\lambda\int_\Omega
f(u_\lambda)T_su\ dx \leq\lambda^\star
s\|f(u_\lambda)\|_{L^1(\Omega)}=C
\|f(u_\lambda)\|_{L^1(\Omega)}^{p'}.
\end{equation}
Combining the previous two inequalities we obtain
$$
\int_{\{u_\lambda>s\}}\left( \frac{4}{p^2}|\nabla_{T,u_\lambda} |\nabla u_\lambda|^{p/2}|^{2}
+ \frac{n-1}{p-1}H_{u_\lambda}^2 |\nabla u_\lambda|^{p} \right) \, dx
\leq
C \|f(u_\lambda)\|_{L^1(\Omega)}^{p'}.
$$
At this point, proceeding exactly as in the proof of Theorem~\ref{Theorem}, we conclude
the $L^r$ estimates established in parts $(a)$ and $(b)$.

In order to prove the $W^{1,p}$-estimate of part $(b)$, recall that
by \eqref{Pohozaev0} we have
$$
\int_\Omega|\nabla u_\lambda|^p\ dx \leq C
\int_{\partial\Omega}|\nabla u_\lambda|^p\ d\sigma. 
$$
Therefore, we need to control the right hand side of the previous inequality.
Since the nonlinearity $f$ is increasing by hypothesis we obtain
$$
f(u_\lambda)\leq f\left(C\|f(u_\lambda)\|_{L^1(\Omega)}^\frac{1}{p-1}\right)
\quad\textrm{in }\Omega_\varepsilon
$$
by \eqref{kkkkeyyyyy}, where $C$ is a constant independent of $\lambda$.

Now, since $-\Delta_p u_\lambda = \lambda f(u_\lambda)\in L^\infty(\Omega_\varepsilon)$ 
in $\Omega_\varepsilon$, it holds
$$
\| u_\lambda \|_{C^{1,\beta} (\overline{\Omega}_\varepsilon)} \leq
C'
$$
for some $\beta\in(0,1)$ by \cite{Lie}, where $C'$ is a constant depending only on 
$n$, $p$, $\Omega$, $f$, and $\|f(u_\lambda)\|_{L^1(\Omega)}$ proving the assertion. 

Finally, assume that $p\geq 2$ and \eqref{convex:assump} holds. From \cite{S}
we know that $f(u^\star)\in L^r(\Omega)$ for all $1\leq r<n/(n-p')$.
In particular, $f(u^\star)\in L^1(\Omega)$. Therefore, parts $(i)$
and $(ii)$ follow directly from $(a)$ and $(b)$.
\end{proof}

\bigskip
\footnotesize
\noindent\textit{Acknowledgments.}
%The authors would like to thank Xavier Cabr\'e for useful
%conversations.
The authors were supported by grant 2009SGR345(Catalunya) and
MTM2011-27739-C04 (Spain). The second author was also supported by
grant MTM2008-06349-C03-01 (Spain).


\begin{thebibliography}{SK}

\normalsize
\baselineskip=17pt

%%%%%%%%%%%%%

\bibitem{A}
Allard, W.K.: {On the first variation of a varifold}.
Ann. Math. {\bf 95},
417--491 (1972)

\bibitem{BBGGPV95}
B\'enilan, Ph., Boccardo, L., Gallou\"et, T.,
Gariepy, R., Pierre, M., V\'azquez, J.L.:
{An $L\sp1$-theory of existence and uniqueness of solutions of nonlinear elliptic equations}.
Ann. Sc. Norm. Super. Pisa, Cl. Sci. {\bf 22},
241--273 (1995)

\bibitem{Cabre09}
Cabr\'e, X.:
{Regularity of minimizers of semilinear elliptic problems up to dimension~4}.
Comm. Pure Appl. Math {\bf 63},
1362--1380 (2010)

\bibitem{CCS09}
Cabr\'e, X., Capella, A., Sanch\'on, M.:
Regularity of radial minimizers of reaction equations involving the $p$-Laplacian.
Calc. Var. PDE {\bf 34},
475--494 (2009)

\bibitem{CS07}
Cabr\'e, X., Sanch\'on, M.:
{Semi-stable and extremal solutions of reaction equations involving the $p$-Laplacian}.
Comm. Pure Appl. Anal. {\bf 6},
43--67 (2007)

\bibitem{CS}
Cabr\'e, X., Sanch\'on, M.:
{Geometric-type Hardy-Sobolev inequalities and applications to the regularity of minimizers}.
Preprint: arXiv:1111.2801v1

\bibitem{CLS} Canino, A., Le, P., Sciunzi, B.:
Local $W_{\rm loc}^{2,m(x)}$ regularity for $p(x)$-Laplace equations.
Preprint

\bibitem{DB} DiBenedetto, E.:
{$C^{1+\alpha}$ local regularity of weak solutions of degenerate elliptic equations}.
Nonlinear Anal. {\bf 7},
827--850 (1983)

\bibitem{DS}Damascelli, L., Sciunzi, B.:
{Regularity, monotonicity and symmetry of positive solutions of
$m$-Laplace equations}. J. Differential Equations {\bf 206},
483--515 (2004)

\bibitem{FSV} Farina, A., Sciunzi, B., Valdinoci, E.:
{Bernstein and De Giorgi type problems: new results via a geometric
approach}. Ann. Sc. Norm. Super. Pisa Cl. Sci. (5) {\bf 7}, 741--791
(2008)

\bibitem{Lie} Lieberman, G.M.:
{Boundary regularity for solutions of degenerate elliptic equations}.
Nonlinear Anal. {\bf 12},
1203--1219 (1988)

\bibitem{MS}
Michael, J.H., Simon, L.:
{Sobolev and meanvalue inequalities on generalized submanifolds of $R^n$}.
Comm. Pure Appl. Math. {\bf 26},
361--379 (1973)

\bibitem{Nedev}
Nedev, G.:
{Regularity of the extremal solution of semilinear elliptic equations}.
C. R. Acad. Sci. Paris S\'er. I Math., {\bf 330},
997--1002 (2000)

\bibitem{Nedev01}
Nedev, G.:
{Extremal solution of semilinear elliptic equations}.
Preprint 2001.

\bibitem{S} Sanch\'on, M.:
{Existence and regularity of the extremal solution of some nonlinear elliptic
problems related to the $p$-Laplacian}.
Potential Anal. {\bf 27},
217--224 (2007)

\bibitem{T} Tolksdorf, P.:
{Regularity for a more general class of quasilinear elliptic equations}.
J. Differential Equations {\bf 51},
126--150 (1984)

\bibitem{Trudinger94} Trudinger, N.S.:
{Isoperimetric inequalities for quermassintegrals}.
Ann. Inst. Henri Poincar\'e {\bf 11},
411--425 (1994)

\bibitem{Trudinger97} Trudinger, N.S.:
{On new isoperimetric inequalities and symmetrization}.
J. reine angew. Math. {\bf 488},
203--220 (1997)

\bibitem{V} V\'azquez, J.L.:
{A strong maximum principle for some quasilinear elliptic equations}.
Appl. Math. Optim. {\bf 12},
191--202 (1984)
\end{thebibliography}
\end{document}